\documentclass{amsart}

\usepackage{graphicx}
\usepackage{scrextend}
\usepackage{amsfonts, amsmath, amssymb,amsthm,comment}  
\usepackage{times,enumerate}
\usepackage{nccmath}
\newenvironment{mpmatrix}{\begin{medsize}\begin{pmatrix}}%
    {\end{pmatrix}\end{medsize}}%
\usepackage{dsfont,bbm}
\usepackage{rotating}
\usepackage{lscape}
\usepackage{url}
\usepackage{multicol}
\usepackage{thmtools, thm-restate}
\usepackage{blkarray}
\usepackage{multicol}
\usepackage[margin=2.5cm]{geometry}
 \usepackage[foot]{amsaddr}
 \usepackage{hyperref}
\usepackage{cancel}

\usepackage{graphicx}
\usepackage[usenames, dvipsnames]{xcolor}
\DeclareGraphicsExtensions{.pdf,.png,.jpg}

\hypersetup{
    colorlinks,
    linkcolor={blue},
    citecolor={blue},
    urlcolor={blue}
}

\DeclareMathOperator{\Span}{span}

\DeclareMathOperator{\Rank}{rank}
\DeclareMathOperator{\cyc}{cyc}

\DeclareMathOperator{\Card}{card}
\DeclareMathOperator{\sign}{sign}

\newcommand{\Sym}{\mathbb{S}}

\newcommand{\RR}{\mathbb R}

\newcommand{\NN}{\mathbb N}

\newcommand{\CC}{\mathbb C}

\newcommand{\SSS}{\mathbb S}

\newcommand{\dd}{{\rm d}}

\newcommand{\cM}{\mathcal M}

\newcommand{\Tr}{\mathrm{Tr}}

\newcommand{\bX}{\mathbb{X}}
\newcommand{\bY}{\mathbb{Y}}
\newcommand{\benu}{\begin{enumerate}}
\newcommand{\eenu}{\end{enumerate}}
\newcommand{\bop}{\begin{opomba}}
\newcommand{\eop}{\end{opomba}}

\newcommand{\tr}{\mathrm{tr}}
\newcommand{\supp}{\mathrm{supp}}

\newtheorem{theorem}{Theorem}[section]
\newtheorem{corollary}[theorem]{Corollary}
\newtheorem{lemma}[theorem]{Lemma}
\newtheorem{proposition}[theorem]{Proposition}

\newtheorem{conjecture}{Conjecture}

\theoremstyle{definition}

\newtheorem{example}[theorem]{Example}

\newcommand{\mc}{\mathcal}
\newcommand{\mbb}{\mathbb}
\newcommand{\mbf}{\mathbf}
\newcommand{\mds}{\mathds}

\newcommand{\ab}[1]{{\color{black} #1}}
\newcommand{\al}[1]{{\color{black} #1}}

\theoremstyle{remark}
\newtheorem{remark}[theorem]{Remark}

\numberwithin{equation}{section}

\begin{document}

\title{The tracial moment problem on quadratic varieties}

\author{Abhishek Bhardwaj${}^1$}
\address{Mathematical Sciences Institute,
The Australian National University,
Union Lane,
Canberra  ACT  2601}
\email{Abhishek.Bhardwaj@anu.edu.au}
\thanks{${}^1$Supported by the Australian Research Council ACEMS IMP Grant (Project ID: CE140100049).}

\author{Alja\v z Zalar${}^2$}
\address{Faculty of Computer and Information Science, University of Ljubljana, Večna pot 113, 1000
Ljubljana, Slovenia}
\email{aljaz.zalar@fri.uni-lj.si}
\thanks{${}^2$Supported by the Slovenian Research Agency grants P1-0288 and J1-8132.}


\date{\today}



\begin{abstract} 
The truncated moment problem asks to characterize finite sequences of real numbers that are the moments of a positive Borel measure on $\RR^n$. 
Its tracial analog is obtained by integrating traces of symmetric matrices and is the main topic of this article.
The solution of the bivariate quartic tracial moment problem with a nonsingular $7\times 7$ moment matrix $\mc M_2$ whose columns are indexed by words of degree 2
was established by Burgdorf and Klep, while in our previos work we completely solved all cases with $\mc M_2$ of rank at most 5, split
$\mc M_2$ of rank 6 into four possible cases according to the column relation satisfied and solved two of them. Our first main result in this article is the solution for $\mc M_2$ satisfying the third possible column relation, i.e., $\bY^2=\mds 1+\bX^2$. Namely, the existence of a representing measure
is equivalent to the feasibility problem of certain linear matrix inequalities.
The second main result is a thorough analysis of the atoms in the measure for $\mc M_2$ satisfying $\bY^2=\mds 1$, 
the most demanding column relation. We prove that size 3 atoms are not needed in the representing measure, a fact proved to be true in all other
cases. The third main result extends the solution for $\mc M_2$ of rank 5 to general $\mc M_n$, $n\geq 2$, with two quadratic column relations. The main technique is the reduction of the problem to the classical univariate truncated moment problem, an approach which applies also in the classical truncated moment problem. Finally, our last main result, which demonstrates this approach, is a simplification of the proof for the solution of the degenerate truncated hyperbolic moment problem first obtained by Curto and Fialkow.

\end{abstract}

\maketitle

\section{Introduction}

The moment problem (MP) is a classical question in analysis which asks when a linear functional can be represented as integration; equivalently, given a sequence of numbers $\beta$, does there exist a positive measure $\mu$ such that $\beta$ represents the moments of $\mu$? This problem is well studied in one dimension (on $\mbb{R}$; see \cite{Akh65,KN77} for instance), while a general solution on $\RR^n$, Haviland's theorem \cite{Hav35}, provides a duality
with positive polynomials and relates the MP to real algebraic geometry (RAG). Renewed interest into the MP in RAG came with Schm{\"u}dgen's solution \cite{Sch91} to the MP over compact semi-algebraic sets; for further results we refer the reader to \cite{Put93,PV99, DP01, PS01, PS06, PS08, Mar08,Lau09}. 
This duality of the MP with positive polynomials has been efficiently used by several authors for 
approximating global optimization problems, most notably Lasserre \cite{Las01, Las09} and Parrilo \cite{Par03},
while recently it has also been useful in understanding solutions of differential equations \cite{MLH11}. There are also many noncommutative generalizations of the MP; the MP for matrix and operator polynomials are considered in \cite{AV03,Vas03,BW11,CZ12,KW13}, the quantum MP in \cite{DLTW08}, free versions of the MP \cite{McC01,Hel02,HM04,HKM12} are the domain of free RAG, while in this paper we are interested in the tracial MP \cite{BK12, BK10}.
	
The multi-dimensional truncated moment problem (TMP), which is more general than the full MP \cite{Sto01}, has been intensively studied in the seminal works of Curto and Fialkow \cite{CF91, CF96, CF98-1, CF98-2,CF08}, with the functional calculus they developed for MP becoming an essential tool for studying moment problems. The bivariate quartic MP is completely solved \cite{CF02, CF04, CF05, CF08,FN10,CS16}, while the sextic has been closely investigated \cite{CFM08, Yoo11,CS15,Fia17}. Recently, the introduction of the core variety provided new results toward the solution of the sextic MP \cite{Fia17,BF+,Sch17,DS18}.
Using convex geometry techniques new sufficient condition for the solvability of the TMP are established also in \cite{Ble15}. 

The truncated tracial moment problem (TTMP), which is the topic of this paper, is the study of linear functionals on the space of non-commutative polynomials that can be represented as traces of 
evaluations on convex combinations of tuples of real symmetric matrices.  It was introduced by Burgdorf and Klep in \cite{BK12, BK10}, where the authors demonstrated its duality with trace-positive polynomials. This duality connects the TTMP to many interesting and important problems such as Connes' embedding conjecture in operator algebras \cite{Con76,KS08-1}, or the now proved BMV conjecture \cite{BMV75,KS08-2, Sta13, Bur11}. Furthermore, \cite{BK12} established tracial analogues of the results of Curto and Fialkow, relating the solution of the TTMP to flat extension of the associated moment matrix 
(see Subsection \ref{BQTMP} for terminology and definitions). For bivariate quartic tracial sequences, an affirmative answer to the TTMP was given in \cite{BK10} when the tracial moment matrix is nonsingular.

Just like the classic TMP, the TTMP is deeply intertwined with optimization of noncommutative polynomials. In \cite{BCKP13} it is shown how minimizing the trace of a noncommutative polynomial evaluated on matrices of some size gives rise to the TTMP. In fact, \cite{BCKP13, BKP16} illustrates how the solution of the TTMP can be used to extract optimizers in this setting.

Inspired by the work of Burgdorf and Klep and Curto and Fialkow, we studied the bivariate quartic TTMP having a singular $(7\times7)$ tracial moment matrix $\mc{M}_{2}$ in \cite{BZ18}. Following the approach of Curto and Fialkow, we analyzed the moment matrix based on its rank, giving a complete classification when the rank is 
at most five.
When the rank is 
six, we reduced the problem to four canonical cases, gave a characterization of when a flat extension exists and in two cases
also proved the existence of a representing measure to be equivalent to the solvability of some linear matrix inequalities.
Moreover we gave explicit examples showing that, unlike in the commutative setting, the existence of a representing measure is mostly \emph{not} equivalent to the existence of a flat extension of the moment matrix. 
 
This article presents new results in the remaining cases of our analysis of the singular quartic bivariate TMP and expands many of the results from degree four to arbitrary degree. We next present the \emph{Bivariate} TTMP and some basic concepts and definitions. We then give an organization of the paper and a summary of our main results.  



\subsection{Bivariate truncated 
tracial moment problem} \label{BQTMP}
 
 In this subsection, we make our problem of study precise and introduce basic definitions used throughout this article. 

\subsubsection{Noncommutative bivariate polynomials}

	We denote by $\left\langle X,Y\right\rangle$ the \textbf{free monoid} generated by 
	the noncommuting letters $X,Y$ and call its elements \textbf{words} in $X, Y$. For a word $w\in \left\langle X,Y\right\rangle$, $w^{\ast}$ is its reverse, and $v\in\left\langle X,Y\right\rangle$ is \textbf{cyclically equivalent} to $w$, which we denote by $\displaystyle{v\overset{\cyc}{\sim} w}$, if and only if $v$ is a cyclic permutation of $w$. 
	
	Consider the free algebra $\mbb{R}\!\left\langle  X,Y\right\rangle$ of polynomials in $X,Y$
	with coefficients in $\mbb{R}$. Its elements are called \textbf{noncommutative (nc) polynomials}. Endow $\mbb{R}\!\left\langle X,Y\right\rangle$ with the involution $p \mapsto p^{*}$ fixing $\mbb{R} \cup \{ X,Y\}$ pointwise. The length of the longest word in a polynomial $f\in \mbb{R}\!\left\langle  X,Y \right\rangle$  is the \textbf{degree} 
of $f$ and is denoted by $\deg(f)$ or $|f|$.  We write $\mbb{R}\!\left\langle  X,Y\right\rangle_{\leq k}$ for
	all polynomials of degree  at most $k$. For a \textbf{nc} polynomial $f$, its \textbf{commutative collapse} $\check{f}$ is obtained by replacing the \textbf{nc} variables $X,Y$, with commutative variables $x,y$, and similarly for words $w\in\langle X,Y\rangle$.
	 
\subsubsection{Bivariate truncated
real tracial moment problem}

Given a sequence of real numbers $\beta\equiv \beta^{(2n)}=(\beta_w)_{|w|\leq 2n}$, indexed by words $w$ of length at most $2n$ such that
	\begin{equation}\label{cyclic-cond}
		\beta_{v}=\beta_{w}\quad\text{whenever } v\overset{\cyc}{\sim} w
		\quad \text{and}\quad \beta_w=\beta_{w^\ast}\quad\text{for all } |w|\leq 2n,
	\end{equation}
 the \textbf{bivariate truncated
real tracial moment problem (BTTMP)} for $\beta$ asks to find conditions for the existence of
 $N\in\mathbb{N}$, $t_i\in\mathbb{N}$, $\lambda_i\in \RR_{> 0}$ with 
$\sum_{i=1}^N \lambda_i=1$ and pairs of real symmetric matrices $(A_{i},B_{i})\in (\mathbb{SR}^{t_i\times t_i})^2$,
such that 
\begin{equation}\label{truncated tracial moment sequence special-biv}
	\beta_{w}=\sum_{i=1}^N \lambda_i \text{Tr}(w(A_{i},B_{i})),
\end{equation}
where $w$ runs over the indices of the sequence $\beta$ and $\text{Tr}$ denotes the \textbf{normalized trace}, i.e., 
	$$
	\text{Tr}(A)=\frac{1}{t}\tr(A)\quad \text{for every } A\in \RR^{t\times t}.
	$$	
If such data exist, we say that $\beta$ admits a representing measure. If $\beta_1=1$, then we say $\beta$ is \textbf{normalized}. 
	We may always assume that $\beta$ is normalized (otherwise we replace $\Tr$ with $\frac{1}{\beta_1}\Tr$). 
 The vectors $(A_i,B_i)$ are  \textbf{atoms of size $t_i$} and the numbers $\lambda_i$ are \textbf{densities}.
We say that $\mu$ is a representing measure of \textbf{type}
$(m_1,m_2,\ldots,m_r)$ if it consists of exactly $m_i\in \NN\cup\{0\}$ atoms of size $i$ and $m_r\neq 0$.
A representing measure of type $(m_1^{(1)},m_2^{(1)},\ldots,m_{r_1}^{(1)})$ is
\textbf{minimal}, if there does not exist another representing measure of type $(m_1^{(2)}$, $m_2^{(2)}$,$\ldots$, $m_{r_2}^{(2)})$ such that
	$$r_2<r_1\quad \text{or} \quad (r:=r_1=r_2\quad \text{and} \quad
	(m_{r}^{(2)},m_{r-1}^{(2)},\ldots,m_{1}^{(2)}) \prec _{\text{lex}} (m_{r}^{(1)},m_{r-1}^{(1)},\ldots,m_{1}^{(1)}),$$
where $\prec _{\text{lex}}$ denotes the usual lexicographic order on $(\NN\cup\{0\})^r$.
We say that $\beta$ admits a \textbf{noncommutative (nc) measure}, if it admits a minimal measure of type 
$(m_1,m_2,\ldots,m_r)$ with $r>1$.
If $\beta_{w}=\beta_{\check{w}}$ for all $w\in\langle X,Y\rangle$, we call $\beta$ a \textbf{commutative (cm) sequence} and the MP reduces to the classical one studied by Curto and Fialkow. Otherwise we call $\beta$ an \textbf{noncommutative (nc) sequence}.

\begin{remark}
\begin{enumerate}
	\item Note that replacing a vector $(A_i,B_i)$ with any 
		vector $$(U_i A_i U_i^t , U_i B_i U_i^{t})\in (\mathbb{SR}^{t_i\times t_i})^2$$			where $U_i\in \RR^{t_i\times t_i}$ is an orthogonal matrix, preserves	
		(\ref{truncated tracial moment sequence special-biv}).
	\item By the tracial version \cite[Theorem 3.8]{BCKP13} of Bayer-Teichmann theorem 
		\cite{BT06}, the problem (\ref{truncated tracial moment sequence special-biv}) is 
		equivalent to the more general problem of finding a probability measure $\mu$
		on $(\mathbb{SR}^{t\times t})^2$ such that 
		$\beta_{w}=\int_{(\mathbb{SR}^{t\times t})^2} \Tr(w(A,B))\; \dd\mu(A,B)$.
\end{enumerate}
\end{remark}

We associate to the sequence $\beta^{(2n)}$ the \textbf{truncated moment matrix  of order $n$}, defined by
	$$\mc{M}_n:=\mc{M}_n(\beta^{(2n)})=(\beta_{w_1^{\ast}w_2})_{|w_1|\leq n,|w_2|\leq n},$$
where the rows and columns are indexed by words in $\RR\!\langle X,Y\rangle_{\leq n}$ in graded lexicographic order with
$X$ being smaller than $Y$, 
e.g., for $n=2$ we have
	$$\mds 1\prec _{\text{lex}} \bX \prec _{\text{lex}} \bY \prec _{\text{lex}} \bX^2 
	\prec _{\text{lex}} \bX\bY \prec _{\text{lex}} \bY\bX \prec _{\text{lex}} \bY^2.$$
\noindent Observe that the matrix $\mc{M}_n$ is symmetric.
The following is a well-known necessary condition for the existence of a measure in the classical commutative moment problem and easily extends to the tracial case.

\begin{proposition}
	\label{Mn-psd}
	If $\beta^{(2n)}$ admits a measure, then $\mc{M}_n$ is positive semidefinite.
\end{proposition}

Let $(X,Y)\in (\mbb{SR}^{t\times t})^{2}$ where $t\in \NN$. We denote by
$\mc{M}^{(X,Y)}_n$ the moment matrix generated by $(X,Y)$, i.e.,
$\beta_{w(X,Y)}=\Tr(w(X,Y))$ for every $|w(X,Y)|\leq 2n$.

\subsection{Results and Readers Guide}
We present the four major contributions in this article.

\subsubsection{TTMP to LMI} Firstly, in \cite[Corollaries 7.6 and 7.9]{BZ18} we proved that the existence of a nc measure for $\mc{M}_2$ of rank 6 satisfying one of the relations $\bY^2=\mds 1-\bX^2$ or $\bX\bY+\bY\bX=\mbf 0$ is equivalent to the feasibility problem of three linear matrix inequalities and a rank-to-cardinality condition (a necessity arising from a cm moment problem).  
A core component of the proof was to show that when 
$\beta_X=\beta_Y=\beta_{X^3}=\beta_{X^2Y}=\beta_{Y^3}=0$  
we have the following result (see \cite[Theorems 7.5 (1), 7.8 (1)]{BZ18}):\\

\begin{addmargin}[2em]{2em}
	\textit{
		For the smallest $\alpha>0$ such that $\Rank\left( \mc{M}_{2} - \alpha W \right) < \Rank \left(\mc{M}_{2}\right)$, the matrix $\mc{M}_{2}-\alpha W$ admits a measure, where $W = \left(\mc M_2^{(1,0)}+\mc M_2^{(-1,0)}\right)$ for $\bY^{2} = \mds{1}-\bX^{2}$ (resp. $W =\mc M_2^{(0,0)}$ for $\bX\bY+\bY\bX=\mbf 0$)}.\\
\end{addmargin}

\noindent Applying the same method of subtracting $\alpha\left(\mc M_2^{(0,1)}+\mc M_2^{(0,-1)}\right)$ in the case of the relation $\bY^2=\mds 1+\bX^2$ does not always work. 

	Nevertheless, in Section \ref{S1} we show that there does in fact exist a matrix $W$ such that the result above \emph{always} holds also for the relation $\bY^2=\mds 1+\bX^2.$
The matrix $W$ is constructed as a sum of moment matrices generated by carefully chosen commutative atoms (see \eqref{subtract-part}). Consequently, we are able to reformulate the existence of a nc measure for  a rank 6 $\mc{M}_{2}$ satisfying the relation $\bY^{2}=\mds{1}+\bX^{2}$, into feasibility problems of LMI's and a rank-to-cardinality condition.

\subsubsection{Size of Atoms} Secondly, in \cite[Proposition 4.1 (2)]{BZ18} we proved that the moment sequence $\beta^{(4)}$ with a moment matrix $\mc M_2$ of rank 6 can always be transformed by using appropriate affine linear transformation 
to a moment sequence 
$\widetilde{\beta}^{(4)}$,  
with $\widetilde{\mc M}_2$ satisfying one of the four canonical relations 
\begin{equation}\label{poss-rel}
\bY^2=\mds 1 - \bX^2,
\quad\text{or}\quad\bX\bY+\bY\bX=\mbf 0,\quad\text{or}\quad\bY^2=\mds 1+\bX^2,
\quad\text{or}\quad\bY^2=\mds 1.
\end{equation} 
In the first three cases we showed that we may assume that the nc atoms 
$(X_i,Y_i)\in \big(\SSS\RR^{t_i'\times t_i'}\big)^2$, $t_{i}'>1$, have an elegant form, i.e., 
\begin{equation}\label{atoms-nice}
X_i=\left(\begin{array}{cc} \gamma_i I_{t_i}& B_i \\ B_i^t & -\gamma_i I_{t_i}\end{array}\right),\quad 
Y_i=\left(\begin{array}{cc} \mu_i I_{t_i}  & 0 \\ 0 & -\mu_i I_{t_i}\end{array}\right),
\end{equation}
where $\gamma_{i}\geq 0$, $\mu_{i}>0$, $B_i$ is a matrix of size $t_i$ (see \cite[Proposition 5.1]{BZ18}). 
Since $Y_i^2=\mu_i^2 I_{t_i'}$, $\mc M_2^{(X_i,Y_i)}$ is of rank at most 5 and hence admits a measure if type
 $(m,1)$, $m\in \{1,2,3\}$, by \cite[\S6]{BZ18}.
In the fourth relation of \eqref{poss-rel} the nc atoms need not be of the form \eqref{atoms-nice}, making this case particularly difficult. In Section \ref{S2} we thoroughly analyze the possible atoms in representing measure, and prove that atoms of size 3 are not needed.


\subsubsection{Extensions to order $n$} Thirdly, in Section \ref{S3} we extend our results from $\mc{M}_{2}$ of rank 5 to $\mc{M}_{n}$ with $n\in\NN$. The main idea is as follows. By first applying an affine linear transformation to $\mc{M}_n$ we may assume that it satisfies the relation 
\begin{equation} \label{first-rel-intro}
  \bX\bY+\bY\bX=\mbf 0,
\end{equation}
and one of the relations
\begin{equation}\label{second-rel-intro}
\bY^2=\mds 1 - \bX^2,
\quad\text{or}\quad\bY^2=\mds 1,
\quad\text{or}\quad\bY^2=\mds 1+\bX^2,
\quad\text{or}\quad\bY^2=\bX^2.
\end{equation} 
Due to \eqref{first-rel-intro}, all the moments $\beta_{X^iY^j}$ with one of the exponents $i,j$ odd and the other nonzero, are equal to zero (see Lemma \ref{zero-moment-general2}).
Additionally, the nc atoms (see Lemma \ref{zero-moment-general}) do not contribute anything to the moments $\beta_X$ and $\beta_Y$, those two must be represented by size 1 atoms in the measure.
There are at most 4 size 1 atoms satisfying \eqref{first-rel-intro} and \eqref{second-rel-intro}, thus there is (under the L{\"o}wner partial ordering) a smallest 
cm matrix $M$ satisfying $\beta_X^M=\beta_X$, $\beta_Y^M=\beta_Y$. 
Subtracting this matrix from $\mc M_n$ we end up with two classical univariate truncated 
	moment problems, one on rows/columns $\{\mds 1, \bX, \bX^2,\ldots,\bX^n\}$ and the other on $\{\mds 1, \bY, \bX\bY, \bX^2\bY,\ldots,\bX^{n-1}\bY\}.$
It turns out that solving the first one also solves the second one due to their connection comming from \eqref{second-rel-intro}.

\subsubsection{Reduction of the TMP on degenerate hyperbolas} Finally, in Section \ref{S4}
we give a simplied proof for the solution of the TMP on degenerate hyperbolas which was discovered by Curto and Fialkow \cite[Theorem 3.1]{CF05}. 
The idea for the proof, inspired by the extension results from Section \ref{S3}, is to reduce the bivariate TMP down to the univariate one.

\begin{remark}
The reduction of the bivariate TMP to the univariate one can also by used in some other cases of the quartic TMP and is also very 
 efficient beyond quadratic column relations. 
We will present this approach in our future work \cite{BZ+} where we study the TMP with column relations of higher degrees.
\end{remark}

\vspace{0.3cm}
\noindent \textbf{Acknowledgement.} The authors would like to thank Igor Klep for insightful discussions and comments on the preliminary versions of this article.

\section{Preliminaries}

In this section we present elementary results for the tracial moment problem and establish some additional notation. Many of these are direct analogues of the corresponding results in the commutative setting.

\subsection{Support of a measure and RG relations} \label{supp-and-RG}
Let $A$ be a matrix with its rows and columns indexed by words in $\mbb{R}\langle X,Y\rangle_{\leq n}$.
For a word $w$ in $\RR\!\langle X,Y\rangle_{\leq n}$ we denote by $w(\mathbb{X,Y})$ the column of $A$ indexed by $w$.
We write $[A]_{\{R,C\}}$  for the compression of $A$ to the rows and columns indexed by elements of $R$ and $C$ resp., with $R,C \subset \RR\langle X,Y\rangle_{\leq n}$ subsets of words. When we have $R=C$, we simply write $[A]_{R}$.
$\mbf{0_{k_1\times k_2}}$ stands for the $k_1\times k_2$ matrix with zero entries.
Usually we will omit the subindex $k_1\times k_2$, where the size will be clear from the context.

Let $\mc{C}_{\mc{M}_n}$ denote the span of the column space of $\mc{M}_n$, i.e., 
	$$
	\mc{C}_{\mc{M}_{n}} = \Span \left\{ w(\bX, \bY) : w\in\mbb{R}\!\left\langle X,Y \right\rangle_{\leq n} \right\}= \Span\left\{ \mds{1},\bX,\bY,\bX^{2},\bX\bY,\bY\bX,\bY^{2},\dotsc,\bX^{n},\dotsc,\bY^{n} \right\}.
	$$
For a polynomial $p\in \mbb{R}\!\langle X,Y\rangle _{\leq n}$ of the form $p=\sum_w a_{w}w(X,Y)$,
we define 
	$$p(\mathbb{X},\mathbb{Y})=\sum_w a_{w} w(\bX,\bY)$$ and notice that 
	$p(\bX,\bY)\in \mc{C}_{\mc{M}_n}$. We express linear dependencies among the columns of $\mc{M}_n$ as 
	$$p_1(\bX,\bY)=\mbf 0,\ldots, p_m(\bX,\bY)=\mbf 0,$$ for some $p_1,\ldots,p_m\in\mbb{R}\!\langle X,Y\rangle _{\leq n}$, $m\in \NN\cup\{0\}$. We define the \textbf{free zero set} $\mc{Z}(p)$ of $p\in\mbb{R}\!\langle X,Y\rangle$ by
	$$\mc{Z}(p):=\left\{ (A,B)\in(\mbb{SR}^{t\times t})^{2} : t\in \NN,\; p(A,B)=\mbf 0_{t\times t}\right\}.$$

\begin{theorem}\cite[Theorem 2.2]{BZ18}\label{support lemma} 
	Suppose $\beta^{(2n)}$ admits a representing measure consisting of finitely many atoms 
	$(X_i,Y_i)\in (\mathbb{SR}^{t_i\times t_i})^2$, $t_i\in \NN$, 
	with the corresponding densities $\lambda_i\in (0,1)$, $i=1,\ldots, r$, $r\in \NN$.
	Let $p\in\mbb{R}\!\left\langle X,Y \right\rangle_{\leq n}$ be a polynomial.	
	Then the following are true:
	\begin{enumerate}
		\item\label{point-1-support} 
			We have 
				$$\bigcup_{i=1}^r\; (X_i,Y_i)\subseteq \mc{Z}(p) \quad \Leftrightarrow \quad 
					p(\bX,\bY)=\mbf 0\; \text{ in }\mc{M}_n.$$
		\item\label{point-2-support}  Suppose the sequence $\beta^{(2n+2)}=(\beta_w)_{|w|\leq n+1}$ is the extension of $\beta$ generated by 
				$$\beta_w=\sum_{i=1}^r \lambda_i \Tr(w(X_i,Y_i)).$$ 
			Let $\mc{M}_{n+1}$ be the corresponding moment matrix.
			Then:
			$$p(\bX,\bY)=\mbf 0\;\text{ in }\mc{M}_n\quad \Rightarrow \quad
				p(\bX,\bY)=\mbf 0 \;\text{ in }\mc{M}_{n+1}.$$
		\item\label{point-3-support}  (Recursive generation) For $q \in\mbb{R}\!\left\langle X,Y \right\rangle_{\leq n}$ such that $pq \in\mbb{R}\!\left\langle X,Y \right\rangle_{\leq n}$, we have
		 		\begin{equation*} \label{pq-lemma}
					p(\bX,\bY)=\mbf{0}\;\text{ in }\mc{M}_n\quad\Rightarrow\quad
					(pq)(\bX,\bY)=(qp)(\bX,\bY)=\mbf{0}\;\text{ in }\mc{M}_n.
				\end{equation*}
	\end{enumerate}
\end{theorem}

Column relations rising in $\mc{M}_n$ through an application of 
	Theorem \ref{support lemma}  (\ref{point-3-support})  are called \textbf{RG relations}.
	If $\mc M_n$ satisfies RG relations, we say $\mc{M}_n$ is \textbf{recursively generated}.
The first consequence of the RG relations is the following important observation about a nc
moment matrix $\mc{M}_n$.

\begin{corollary}\cite[Colloralies 2.3, 2.4]{BZ18} \label{lin-ind-of-4-col}
	Suppose $n\geq 2$ and $\beta^{(2n)}$ be a sequence such that 
	$\beta_{X^2Y^2}\neq \beta_{XYXY}$. 
	Then the columns $\mds 1, \bX, \bY, \bX\bY$ of $\mc{M}_n$ are linearly independent.
	Hence, if  $\mc{M}_n$ is of rank at most 3 with $\beta_{X^2Y^2}\neq \beta_{XYXY}$, then $\beta$ does not admit a representing measure.
\end{corollary}



\subsection{Flat extensions} \label{flat-ext-prel}

For a matrix $A\in \Sym\RR^{s\times s}$, an \textbf{extension} 
$\widetilde{A}\in \Sym\RR^{(s+u)\times (s+u)}$ of the form
	$$\widetilde A=\begin{pmatrix} A & B \\ B^t & C \end{pmatrix}$$
for some $B\in \RR^{s\times u}$ and $C\in \RR^{u\times u}$, is called $\textbf{flat}$ if 
$\Rank(A)=\Rank(\widetilde A)$. By a result of \cite{Smu59}, this is equivalent to saying that there is a matrix
$W\in \RR^{s\times u}$ such that $B=AW$ and $C=W^t A W$. Flat extension provide an approach to solving the BTTMP via the following.

\begin{theorem} \cite[Theorem 3.19]{BK12} \label{flat-meas}
		Let $\beta\equiv\beta^{(2n)}$ be a sequence satisfying (\ref{cyclic-cond}). 
	If $\mc M_n(\beta)$ is psd and is a flat extension of $\mc M_{n-1}(\beta)$, then $\beta$ admits a
	representing measure.
\end{theorem}

\subsection{Riesz functional and affine linear transformations} \label{affine linear-trans}

Any sequence $\beta^{(2n)}$ which satisfies \eqref{cyclic-cond}
defines the \textbf{Riesz functional} 
$L_{\beta^{(2n)}}:\mbb{R}\!\left\langle  X,Y\right\rangle_{\leq 2n}\to \RR$ by
$$
	L_{\beta^{(2n)}}(p):=\sum_{|w|\leq 2n} a_{w}\beta_{w},\quad \text{where }p=\sum_{|w|\leq 2n} a_{w}w.
$$
Notice that 
	$$\beta_w=L_{\beta^{(2n)}}(w)\quad \text{for every }|w|\leq 2n.$$
An important result for converting a given moment problem into a simpler, equivalent one is the application of affine linear transformations to a sequence $\beta$. For non-commuting letters $X,Y$ and $a,b,c,d,e,f\in \RR$ with $bf-ce \neq 0$, let us define 
  \begin{equation}\label{trans-form}
	\phi(X,Y)=(\phi_1(X,Y),\phi_2(X,Y)):=(a +bX+cY,d +eX+fY).
  \end{equation}
Let $\widetilde \beta^{(2n)}$ be the sequence obtained by the rule
  \begin{equation}\label{lin-trans-orig}
	\widetilde \beta_{w}=L_{\beta^{(2n)}}(w\circ\phi(X,Y))\quad \text{for every }|w|\leq n.
  \end{equation}
Notice that 
	$$L_{\widetilde\beta^{(2n)}}(p) =L_{\beta^{(2n)}}(p\circ \phi(X,Y))\quad \text{for every }p \in \mbb{R}\!\langle X,Y\rangle_{\leq n}.$$
For a polynomial $p\in \mbb{R}\!\langle X,Y\rangle_{\leq 2n}$ let $\widehat{p}=(a_{w})_{w}$ be its coefficient vector with respect to the  lexicographically-ordered words in $\RR\!\left\langle X,Y\right\rangle_{\leq 2n}$.
The following proposition allows us to make affine linear changes of variables.

\begin{proposition}\cite[Proposition 2.6]{BZ18} \label{linear transform invariance-nc}
	Suppose $\beta^{(2n)}$ and $\widetilde{\beta}^{(2n)}$ are as above with the corresponding moment matrices 			
	$\mc{M}_n$ and $\widetilde{\mc{M}}_n$, respectively.
	Let $J_\phi: \RR\!\left\langle X,Y\right\rangle_{\leq 2n}\to \RR\!\left\langle X,Y\right\rangle_{\leq 2n}$ be 	the linear map given by 
		$$J_\phi \widehat p:=\widehat{p\circ \phi}.$$
	Then the following hold:
\begin{enumerate}
	\item $\widetilde{\mc{M}}_n=(J_\phi)^t \mc{M}_n J_\phi.$
	\item $J_\phi$ is invertible.
	\item  $\widetilde{\mc{M}}_n\succeq 0 \Leftrightarrow  \mc{M}_n\succeq 0.$
	\item $\Rank (\widetilde{\mc{M}}_n)=\Rank (\mc{M}_n).$
	\item\label{invariance-point5} The formula $\mu =\tilde \mu \circ\phi$ establishes a one-to-one correspondence between the sets of 			representing measures of $\beta$ and $\tilde \beta$, 
		and $\phi$ maps $\supp(\mu)$ bijectively onto $\supp(\tilde \mu)$.
	\item \label{flat-ext-M_n} $\mc{M}_n$ admits a flat extension if and only if $\widetilde{\mc{M}}_n$ admits a flat extension.
\end{enumerate}
\end{proposition}

\section{$\mc M_2$ of rank 6 with relation $\bY^2=\mds 1+\bX^2$} \label{S1}
 
We show in this section that for $\mc{M}_{2}$ of rank 6 which satisfies the relation $\bY^{2}=\mds{1}+\bX^{2}$, the existence of a representing measure is equivalent to the feasibility of  three LMI's, and a rank to cardinality condition.
 
\begin{theorem} \label{M(2)-bc3-r6-new1-cor}
	Suppose $\beta\equiv \beta^{(4)}$ is a normalized nc sequence with a moment matrix $\mc{M}_2$ of rank 6 satisfying the relation
	$\bY^2=\mds 1+\bX^2$. 
	Let $L(a,b,c,d,e)$ be the following linear matrix polynomial
\begin{equation*} 
  \begin{blockarray}{cccccccc}
    & \mds{1}&\bX&\bY &\bX^2&\bX\bY&\bY\bX&\bY^{2}\\
  \begin{block}{c(ccccccc)}
   \mds 1& 	a & \beta_X & \beta_Y & b & c & c &  a+b		\\
   \bX & \beta_X & b & c & \beta_{X^3} & \beta_{X^2Y} & \beta_{X^2Y} & \beta_X+\beta_{X^3} \\
   \bY & \beta_{Y}& c & a+b & \beta_{X^2Y} & \beta_X+\beta_{X^3} & \beta_X+\beta_{X^3} & 
				\beta_Y+\beta_{X^2Y}			\\
   \bX^2 & b & \beta_{X^3} & \beta_{X^2Y} & d & e & e & b+d	\\
   \bX\bY & c & \beta_{X^2Y} & \beta_X+\beta_{X^3} & e & b+d & b+d & c+e					\\
   \bY\bX & c & \beta_{X^2Y} & \beta_X+ \beta_{X^3}& e & b+d & b+d & c+e 					\\
   \bY^2  & a+b & \beta_X+\beta_{X^3} & \beta_{Y}+\beta_{X^2Y} & b+d & c+e & c+e & a+2b+d\\
\end{block}
\end{blockarray},
\end{equation*}
	where $a,b,c,d,e\in \RR$.
	Then $\beta$ admits a nc measure if and only if there exist $a,b,c,d,e\in \RR$ such that
	\begin{enumerate}
		\item\label{point1-bc1-r6} $L(a,b,c,d,e)\succeq 0$,
		\item\label{point2-bc1-r6} $\mc{M}_2-L(a,b,c,d,e)\succeq 0$,
		\item\label{point3-bc1-r6} 
			$(\mc{M}_2-L(a,b,c,d,e))_{\{\mds 1, \bX, \bY, \bX\bY\}} \succ 0$,
		\item\label{point4-bc1-r6} $L(a,b,c,d,e)$ is recursively generated and $\Rank(L(a,b,c,d,e))\leq \Card \mathcal V_L$, where 
		 	$$\mathcal V_L:=\displaystyle\bigcap_{
			\substack{g\in \RR[X,Y]_{\leq 2},\\ 
						g(\bX,\bY)=\mbf 0\;\text{in}\;L(a,b,c,d,e)}}
				\left\{ (x,y)\in \RR^2\colon g(x,y)=0 \right\}.$$ 
	\end{enumerate}
	If $\beta$ admits a measure, then there exists a measure of type $(m,1)$, $m\in \{2,3,4,5\}$.

	In particular, $a,b,c,d,e$ satisfying \eqref{point1-bc1-r6}-\eqref{point4-bc1-r6} exist if 
	  \begin{equation}\label{suff-cond-meas}
	    \beta_{X}=\beta_Y=\beta_{X^3}=\beta_{X^2Y}=\beta_{Y^3}=0.
	  \end{equation}
\end{theorem}
	
Before proving Theorem \ref{M(2)-bc3-r6-new1-cor} we need some auxiliary results.
The form of $\mc{M}_2$ is given by the following proposition.

\begin{proposition}
	Let $\beta\equiv \beta^{(4)}$ be a nc sequence with a moment matrix $\mc{M}_2$ satisfying the relation
		\begin{equation}\label{r6-rel-bc3-eq} 
			\bY^2=\mds{1}+\bX^2.
		\end{equation}
	Then $\mc{M}_2$ is of the form
		\begin{equation}\label{bc3-r6}
	\begin{mpmatrix}
 		\beta_{1} & \beta_{X} & \beta_{Y} & \beta_{X^2} & \beta_{XY} & \beta_{XY} & 			
			\beta_{1}+\beta_{X^2} \\
 		\beta_{X} & \beta_{X^2} & \beta_{XY} & \beta_{X^3} & \beta_{X^2Y} & \beta_{X^2Y} & 		
			\beta_{X}+\beta_{X^3} \\
		 \beta_{Y} & \beta_{XY} & \beta_{1}+\beta_{X^2} & \beta_{X^2Y} & \beta_{X}+\beta_{X^3} & 
		 	\beta_{X}+\beta_{X^3} & \beta_{Y}+\beta_{X^2Y} \\
 		\beta_{X^2} & \beta_{X^3} & \beta_{X^2Y} & \beta_{X^4} & \beta_{X^3Y} & 		\beta_{X^3Y} & \beta_{X^2}+\beta_{X^4} \\
		\beta_{XY} & \beta_{X^2Y} & \beta_{X}+\beta_{X^3} & \beta_{X^3Y} & 
			\beta_{X^2}+\beta_{X^4} & \beta_{XYXY} & \beta_{XY}+\beta_{X^3Y} 
			\\
		 \beta_{XY} & \beta_{X^2Y} & \beta_{X}+\beta_{X^3} & \beta_{X^3Y} & \beta_{XYXY} & 		
		 	\beta_{X^2}+\beta_{X^4} & \beta_{XY}+\beta_{X^3Y} \\
		 \beta_{1}+\beta_{X^2} & \beta_{X}+\beta_{X^3} & \beta_{Y}+\beta_{X^2Y} & 
			\beta_{X^2}+\beta_{X^4} & 
		\beta_{XY}+\beta_{X^3Y} & \beta_{XY}+\beta_{X^3Y} & \beta_{1}+2 \beta_{X^2}+\beta_{X^4} 
	\end{mpmatrix}.
		\end{equation}
\end{proposition}

\begin{proof}
	This is an easy computation using the relation (\ref{r6-rel-bc3-eq}).
\end{proof}

\begin{lemma}\label{linear-trans}
Suppose $\beta\equiv \beta^{(4)}$ is a normalized nc sequence with a positive semidefinite and recursively generated 
moment matrix $\mc M_2$ of rank 5 satisfying the relations
	\begin{equation}\label{starting-relation}
		\bY^2=\mds 1+\bX^2, \quad a\mds 1+ d\bX^2+e(\bX\bY+\bY\bX)=\mbf{0}, 
	\end{equation}
for some $a,d,e\in \RR$ which are not all zero. 
Then there is a linear transformation of the form
	\begin{equation}\label{lin-trans1}
		\phi(X,Y)=(\phi_1(X,Y),\phi_2(X,Y)):=(bX+cY,eX+fY),
	\end{equation}
where $b,c,e,f\in \RR$ satisfy $bf-ce \neq 0$, such that the sequence 
$\widetilde{\beta}^{(4)}$ obtained by the rule \eqref{lin-trans-orig} has a 
a moment matrix $\widetilde{\mc M}_2$ satisfying the relation 
	\begin{equation}\label{first-rel}
	\bX\bY+\bY\bX=\mbf 0
	\end{equation}
and one of the relations
	\begin{equation}\label{second-rel}
	  \bY^2=\mds 1\quad \text{or}\quad \bX^2+\bY^2=\mds 1
	  \quad \text{or}\quad \bY^2-\bX^2=\mds 1.
	\end{equation}
\end{lemma}
	
\begin{proof}

	We separate two cases according to $e$ in \eqref{starting-relation}.\\

	\noindent \textbf{Case 1: }$e=0$.\\
	
	First note that $d\neq 0$ in \eqref{starting-relation}, 
	otherwise $a\mds 1=\mbf 0$ for $a\neq 0$ which is a contradiction since $\mds 1\neq \mbf 0$ ($\beta_1=1$). Hence we
	can rewrite \eqref{starting-relation} as $\bX^2=\widetilde{a}\mds 1$ where $\widetilde{a}\neq 0$.  
	Therefore $\bY^2=(1+\widetilde{a})\mds 1$. 
	Since $\mc M_2$ is psd with a nonzero column $\bX$ (otherwise $\Rank \mc M_2<5$),
	it follows that $0<[\mc M_2]_{\{\bX\}}= \beta_{X^2}$.
	Thus also the column $\bX^2$ is nonzero (since it contains $\beta_{X^2} $), 
	which implies by $\mc M_2$ being psd that  $0< [\mc M_2]_{\{\bX^2\}}=\beta_{X^4}$.
	Hence from $0<\beta_{X^4}=\widetilde{a}\beta_{X^2}$,
	it follows that $\widetilde{a}>0$.
	Now applying the transformation
	$$\phi(X,Y)= \left( \frac{X}{2\sqrt{\widetilde a}}+\frac{Y}{2\sqrt{1+\widetilde a}} ,
			 \frac{Y}{2\sqrt{1+\widetilde a}}-\frac{X}{2\sqrt{\widetilde a}} \right)$$
	to the moment sequence $\beta_{w},$ we get a 
	moment sequence $\widetilde{\beta}_{w}$ with a moment matrix
	$\widetilde{\mc M}_2$ of rank 5 satisfying the relations
	\begin{equation}\label{rel0}
		\bX\bY+\bY\bX=\mbf 0,\quad \bX^2+\bY^2=\mds 1.
	\end{equation}
  	\vspace{0pt}

	\noindent \textbf{Case 2: }$e\neq 0$.\\
 
	Given the starting relations \eqref{starting-relation} we are in Case 2.4 in the proof of 
	\cite[Proposition 4.1 (1)]{BZ18}. Following the proof we see that
	after using only transformations of the form \eqref{trans-form} we end up with a moment sequence 
	$\widetilde{\beta}^{(4)}$ such that $\widetilde{\mc M}_2$ satisfies the relations \eqref{first-rel} and 
	\eqref{second-rel}. Precise transformations can be found in Appendix \ref{append-transforms}. 
\end{proof}

\begin{lemma}\label{linear-trans-on-zero-moments}
Suppose $\beta\equiv \beta^{(4)}$ is a nc sequence satisfying 
	$$\beta_{X}=\beta_Y=\beta_{X^3}=\beta_{X^2Y}=\beta_{Y^3}=0.$$
Let $\phi$ be a linear transformation defined by
	\begin{equation}\label{lin-trans}
		\phi(X,Y)=(\phi_1(X,Y),\phi_2(X,Y)):=(bX+cY,eX+fY),
	\end{equation}
where $b,c,e,f\in \RR$ satisfy $bf-ce \neq 0$. 
The sequence $\widetilde{\beta}^{(4)}$ obtained by the rule \eqref{lin-trans-orig} also satisfies
	$$\widetilde{\beta}_{X}=\widetilde{\beta}_Y=\widetilde{\beta}_{X^3}=
		\widetilde{\beta}_{X^2Y}=\widetilde{\beta}_{Y^3}=0.$$
\end{lemma}
	
\begin{proof}
	This is an easy direct calculation. The details can be found in Appendix \ref{calc-for-lemma-zero-moments}.
\end{proof}

The following theorem characterizes normalized nc sequences $\beta$ with a moment matrix $\mc{M}_2$ of rank 6 
satisfying the relation $\bY^2=\mds 1+\bX^2$, which admit a nc measure.

\begin{theorem} \label{M(2)-bc3-r6-new1} 
	Suppose $\beta\equiv \beta^{(4)}$ is a normalized nc sequence with a moment matrix $\mc{M}_2$ of rank 6 satisfying the relation
	$\bY^2=\mds 1+\bX^2$. Then $\beta$ admits a nc measure if and only if $\mc{M}_2$ is positive semidefinite and
	one of the following is true:
		\begin{enumerate}
			\item \label{M(2)-c3-r6-pt1} $\beta_{X}=\beta_Y=\beta_{X^3}=\beta_{X^2Y}=\beta_{Y^3}=0$.
				In this case there exists a nc measure of type $(m,1)$, $m\in \NN$.
			\item \label{M(2)-c3-r6-pt2}
			There exist 
				$$a_1\in (0,1),\quad 	
				a_2\in \left(-2\sqrt{a_1(1+a_1)}, 2\sqrt{a_1(1+a_1)}\right)$$
				such that 
					$$M:=\mc{M}_2-\xi\mc{M}^{(X,Y)}_2$$ 
				is a positive semidefinite, recursively generated cm moment matrix satisfying 
				$$\Rank M\leq \Card \mathcal V_M:=
				\displaystyle\bigcap_{\substack{g\in \RR[X,Y]_{\leq 2},\\ 
						g(\bX,\bY)=\mbf 0\;\text{in}\;M}}
					\left\{ (x,y)\in \RR^2\colon g(x,y)=0 \right\},$$
				where
					\begin{equation}\label{(X,Y)-form}
						X=\begin{pmatrix} \sqrt{a_1} & 0 \\
										0 & -\sqrt{a_1} \end{pmatrix},\quad
						Y=\sqrt{(1+a_1)} \begin{pmatrix} \frac{a}{2} & \frac{1}{2}\sqrt{4-a^2}\\
							\frac{1}{2}\sqrt{4-a^2} & -\frac{a}{2}
							\end{pmatrix},
					\end{equation}	
				with $\displaystyle a=\frac{a_2}{\sqrt{a_1(1+a_1)}},$
				and $\xi>0$ is the smallest positive number such that  
				$$\Rank{\left(\mc{M}_2-\xi\mc{M}^{(X,Y)}_2\right)}< \Rank{\mc{M}_2}.$$
		\end{enumerate}
	Moreover, if $\beta$ admits a measure, then there exists a measure of type $(m,1)$, $m\in \{2,3,4,5\}$.
\end{theorem}

\begin{proof}
	First we will prove (1).
	In this case $\mc{M}_2$ is of the form
		$$\begin{mpmatrix}
		1 & 0 & 0 & \beta_{X^2} & \beta_{XY} & \beta_{XY} &  1+\beta_{X^2}		\\
		0 & \beta_{X^2} & \beta_{XY} & 0 & 0 & 0 & 0 					\\
		0& \beta_{XY} & 1+\beta_{X^2} & 0 & 0 & 0 & 0			\\
		\beta_{X^2} & 0 & 0 & \beta_{X^4} & \beta_{X^3Y} & \beta_{X^3Y} & 
						\beta_{X^2}+\beta_{X^4}	\\
		\beta_{XY} & 0 & 0 & \beta_{X^3Y} & \beta_{X^2}+\beta_{X^4} & \beta_{XYXY} & \beta_{XY}+\beta_{X^3Y}					\\
	\beta_{XY} & 0 & 0 & \beta_{X^3Y} & \beta_{XYXY} & \beta_{X^2}+\beta_{X^4} & 
		\beta_{XY}+\beta_{X^3Y}				\\
		1+\beta_{X^2} & 0 & 0 & \beta_{X^2}+\beta_{X^4} & \beta_{XY}+\beta_{X^3Y}	 & 
	\beta_{XY}+\beta_{X^3Y}	 & 1+2\beta_{X^2}+\beta_{X^4}\end{mpmatrix}.$$
	We define the matrix function
		\begin{equation}\label{subtract-part}
		B(\alpha,\gamma):=
			\mc{M}_2-\alpha \big(\mc{M}_2^{(\gamma,\sqrt{1+\gamma^2})}+
			\mc{M}_2^{(-\gamma,\sqrt{1+\gamma^2})}+
			\mc{M}_2^{(\gamma,-\sqrt{1+\gamma^2})}+\mc{M}_2^{(-\gamma,-\sqrt{1+\gamma^2})}\big),
		\end{equation}
	which is equal to
		$$
			B(\alpha,\gamma)=\begin{mpmatrix}
					 1-4\alpha & 0 & 0 & \beta_{X^2}-4\alpha \gamma^2 & \beta_{XY} & \beta_{XY} & D	\\								0 & \beta_{X^2}-4\alpha \gamma^2 & \beta_{XY} & 0 & 0 & 0 & 0 					\\
					0& \beta_{XY} & D & 0 & 0 & 0 & 0				\\
				\beta_{X^2}-4\alpha \gamma^2 & 0 & 0 & \beta_{X^4}-4\alpha\gamma^4 & \beta_{X^3Y} & \beta_{X^3Y} &  C
						\\
					\beta_{XY} & 0 & 0 & \beta_{X^3Y} & C & E & \beta_{XY}+\beta_{X^3Y}				\\
					\beta_{XY} & 0 & 0 & \beta_{X^3Y} & E &  
						C & \beta_{XY}+\beta_{X^3Y}			\\
					D & 0 & 0 & C & \beta_{XY}-\beta_{X^3Y}	 & \beta_{XY}-\beta_{X^3Y}	 & D+C
				\end{mpmatrix},$$
	where 
		$$C=\beta_{X^2}+\beta_{X^4}-4\alpha \gamma^2(1+\gamma^2),\quad 
		D=1+\beta_{X^2}-4\alpha(1+\gamma^2),\quad 
		E=\beta_{XYXY}-4\alpha \gamma^2(1+\gamma^2).$$
	\vspace{0pt}

	\noindent{\textbf{Claim.}} There exist $\alpha_0>0$ and $\gamma_0>0$ such that
	$B(\alpha_0,\gamma_0)$ is psd and satisfies the column relations
		\begin{equation}\label{B-relations}
			a\mds 1+ d\bX^2+e(\bX\bY+\bY\bX)=\mbf{0}, \quad \bY^2=\mds 1+\bX^2
		\end{equation}
	for some $a,d,e\in \RR$ which are not all zero. 
	Let $\beta_w^{(\alpha_0,\gamma_0)}$ be the moments of $B(\alpha_0,\gamma_0)$.
	Then:
		\begin{equation}\label{B-moments}
			\beta^{(\alpha_0,\gamma_0)}_{X}=\beta^{(\alpha_0,\gamma_0)}_Y=
			\beta^{(\alpha_0,\gamma_0)}_{X^3}=\beta^{(\alpha_0,\gamma_0)}_{X^2Y}=		
			\beta^{(\alpha_0,\gamma_0)}_{XY^2}=\beta^{(\alpha_0,\gamma_0)}_{Y^3}=0.
		\end{equation}
	\vspace{0pt}

	Since
		$$\det\big([B(\alpha,\gamma)]_{\{ \bX, \bY \}}\big)=
		16a^2(1+\gamma^2) \alpha^2+(-4\gamma^2-(4+2\gamma^2)\beta_{X^2})\alpha+
		(\beta_{X^2}^2+\beta_{X^2}-\beta_{XY}^2)
		$$
	is quadratic in $\alpha$, we have that the equation $\det\big([B(\alpha,\gamma)]_{\{ \bX, \bY \}}\big)=0$
	has solutions
		$$
		\alpha_{1,2}=\frac{\gamma^2+\beta_{X^2}+2\gamma^2\beta_{X^2}\pm
		\sqrt{(\gamma^2-\beta_{X^2})^2+ 4\gamma^2\beta_{XY}^2(1+\gamma^2)}}
		{8\gamma^2(1+\gamma^2)}.
		$$
	Since
		$$\det\big([B(\alpha,\gamma)]_{\{ \bX\bY, \bY\bX \}}\big)=
		8\gamma^2(1+\gamma^2)(\beta_{XYXY}-\beta_{X^4}-\beta_{X^2})\alpha-
		(\beta_{XYXY}+\beta_{X^4}+\beta_{X^2})
		(\beta_{XYXY}-\beta_{X^4}-\beta_{X^2})
		$$
	is linear in $\alpha$ 
	and $[\mc{M}_2]_{\{ \bX\bY, \bY\bX \}}$ is positive definite, this implies that
		$$0<\det\big([\mc{M}_2]_{\{ \bX\bY, \bY\bX \}}\big)=
		-(\beta_{XYXY}+\beta_{X^4}+\beta_{X^2})
		(\beta_{XYXY}-\beta_{X^4}-\beta_{X^2})$$
	and in particular $\beta_{XYXY}-\beta_{X^4}-\beta_{X^2}\neq 0$,
	the equation $\det\big([B(\alpha,\gamma)]_{\{ \bX\bY, \bY\bX \}}\big)=0$
	has a solution
		$$\alpha_3=\frac{\beta_{XYXY}+\beta_{X^2}+\beta_{X^4}}{8\gamma^2(1+\gamma^2)}.$$
%
	\vspace{0pt}	

	\noindent \textbf{Subclaim.}
	For $\gamma$ big enough it is true that $\alpha_3<\min\big(\alpha_{1},\alpha_2,\frac{1}{4}\big)$.\\
	
	We separate two cases: $\beta_{XY}=0$ and $\beta_{XY}\neq 0$.\\
	
	\noindent \textbf{Case 1: }$\beta_{XY}=0$.\\
	
	For $\gamma>0$ such that $\gamma^2\geq \beta_{X^2}$, $\alpha_1$ and $\alpha_2$ are equal to
	$$\alpha_{1}=\frac{2\gamma^2+2\gamma^2\beta_{X^2}}
	{8\gamma^2(1+\gamma^2)}=\frac{1+\beta_{X^2}}{4(1+\gamma^2)},\quad
	\alpha_{2}=\frac{2(1+\gamma^2)\beta_{X^2}}
	{8\gamma^2(1+\gamma^2)}=\frac{\beta_{X^2}}{4\gamma^2}.
	$$
	Since $\alpha_3$ has $\gamma^4$ in the denominator, it is smaller than $\alpha_{1}, \alpha_2$ and $\frac{1}{4}$ for $\gamma$ big enough. \\
	
	\noindent \textbf{Case 2: }$\beta_{XY}\neq 0$.\\

	Calculating the limits of $\alpha_1$ and $\alpha_2$ where $\gamma$ goes to $\infty$ we get
	\begin{align*}
	\lim_{\gamma\to\infty}\alpha_1
	&= \lim_{\gamma\to\infty}\frac{\gamma^2(1+2\beta_{X^2})+
		\gamma^2\sqrt{(1+4\beta_{XY}^2)}}{8\gamma^2(1+\gamma^2)}=
		\lim_{\gamma\to\infty}\frac{(1+2\beta_{X^2})+
		\sqrt{(1+4\beta_{XY}^2)}}{8(1+\gamma^2)},\\
	\lim_{\gamma\to\infty}\alpha_2
	&= \lim_{\gamma\to\infty}\frac{\gamma^2(1+2\beta_{X^2})-
		\gamma^2\sqrt{(1+4\beta_{XY}^2)}}{8\gamma^2(1+\gamma^2)}=
		\lim_{\gamma\to\infty}\frac{(1+2\beta_{X^2})-
		\sqrt{(1+4\beta_{XY}^2)}}{8(1+\gamma^2)}.
	\end{align*}
	Since $[\mc{M}_2]_{\{\bX,\bY\}}$ is positive definite, it follows that
	$\det ([\mc{M}_2]_{\{\bX,\bY\}})>0$, i.e., 
		$$\beta_{XY}^2<(1+\beta_{X^2})\beta_{X^2}.$$
	Hence,
		$$1+4\beta_{XY}^2< 1+4(1+\beta_{X^2})\beta_{X^2}=
		(1+2\beta_{X^2})^2.$$
	Therefore, the numerators in $\alpha_1, \alpha_2$ are strictly positive.
	Therefore for $\gamma$ big enough, $\alpha_3$ is smaller than $\alpha_1, \alpha_2$ and $\frac{1}{4}$, since it has $\gamma^4$ in the denominator. 
	This proves the subclaim.\\

	Let us now fix $\gamma_0$ big enough such that  $\alpha_3$ is smaller than 
	$\alpha_{1},\alpha_2$. Let $\alpha_0>0$ be the smallest positive number such that the rank of 	
	$B(\alpha_0,\gamma_0)$ is smaller than 6. Since $B(0,\gamma_0)$ is psd of rank 6, $B(\alpha_0,\gamma_0)$
	is also psd of rank at most 5. Since in particular,
		$[B(\alpha_0,\gamma_0)]_{\{ \bX\bY, \bY\bX \}}$ is psd, it follows that
	$\alpha_0\leq \alpha_3$.
	From the subclaim we conclude that  $\alpha_0<\min(\alpha_{1},\alpha_2,\frac{1}{4})$.
	Using this and the form of  $B(\alpha_0,\gamma_0)$ we conclude that 
	$B(\alpha_0,\gamma_0)$ satisfies \eqref{B-relations} and \eqref{B-moments} which proves Claim.\\

	The rank of $B(\alpha_0,\gamma_0)$ is at least 4 since the columns $\mds 1, \bX, \bY, \bX\bY$ are linearly independent.
	Indeed, the submatrix 
		$$[B(\alpha_0,\gamma_0)]_{\{\mds 1, \bX, \bY\}}=
		[B(\alpha_0,\gamma_0)]_{\{\mds 1\}}\oplus 
		[B(\alpha_0,\gamma_0)]_{\{\bX, \bY\}}$$ 
	is block diagonal. By the above $\det\big( [B(\alpha_0,\gamma_0)]_{\{\bX, \bY\}}\big)\neq 0$. Since $\alpha_0\leq \alpha_3 < \frac{1}{4}$, 
	$[B(\alpha_0,\gamma_0)]_{\{\mds 1\}}\neq 0$ and the column $\mds 1$ is nonzero. Hence the columns $\mds 1, \bX, \bY$ are linearly independent. 
	Note also that in the full matrix $B(\alpha_0,\gamma_0)$, $\bX\bY$ cannot be a linear combination of
	$\mds 1, \bX, \bY$ 
	since it is not symmetric in rows $\bX\bY$ and $\bY\bX$.

	Now we separate two cases according to the rank of $B(\alpha_0,\gamma_0)$.
	
	\noindent\textbf{Case 1:} $\Rank B(\alpha_0,\gamma_0)=4.$ 
	By the form of $B(\alpha_0,\gamma_0)$ the relations are 
		$$\bX^2=a_1 \mds 1,\quad \bY\bX=a_2 \mds 1 - \bX\bY, \quad \bY^2=(1+a_1)\mds 1$$
	for some $a_1, a_2\in \RR\backslash \{0\}.$
	By \cite[Theorem 3.1 (3)]{BZ18} the measure for the sequence $\widetilde{\beta}^{(\alpha_0,\gamma_0)}_w$ exists 
	and is of type $(0,1)$.\\

	\noindent\textbf{Case 2:} $\Rank B(\alpha_0,\gamma_0)=5.$  	
	By Lemma \ref{linear-trans} there is a transformation of the form \eqref{lin-trans1} which we apply to get
	a moment sequence $\widetilde{\beta}^{(\alpha_0,\gamma_0)}_{w}$ 
	 such that the corresponding moment matrix
	$\widetilde{\mc M}_2$ satisfes the relations \eqref{first-rel} and \eqref{second-rel}. 
	By Lemma \ref{linear-trans-on-zero-moments} in both cases we have that 
		$$\widetilde{\beta}^{(\alpha_0,\gamma_0)}_{X}=
		\widetilde{\beta}^{(\alpha_0,\gamma_0)}_Y=
		\widetilde{\beta}^{(\alpha_0,\gamma_0)}_{X^3}=
		\widetilde{\beta}^{(\alpha_0,\gamma_0)}_{X^2Y}=
		\widetilde{\beta}^{(\alpha_0,\gamma_0)}_{Y^3}=0.$$
	Furthermore, since the rank of $B(\alpha_0,\gamma_0)$ is 5, a measure also exists and is of type $(m_1,1)$ where $m_1\in\{1,2,3\}$
	by \cite[Theorems 6.5, 6.8, 6.11, 6.14]{BZ18}.
	Hence $\beta$ admits a measure of type $(m,1)$, $m\in \NN$. This proves \eqref{M(2)-c3-r6-pt1}.

	It remains to prove \eqref{M(2)-c3-r6-pt2}. Suppose that $\beta$ admits a nc measure. 
	Using Theorem \ref{M(2)-bc3-r6-new1} \eqref{M(2)-c3-r6-pt1} together with  \cite[Proposition 7.3]{BZ18} (note that the result and proof hold in the case of $\bY^2=\mds 1+ \bX^2$ as well), we obtain
		\begin{equation} \label{r6-M(2)-with-rank4-bc4-new22} 
			\mc{M}_2=\sum_{i=1}^m \lambda_i \mc{M}^{(x_i,y_i)}_2 +  \xi \mc{M}^{(X,Y)}_2,
		\end{equation}
	where $(x_i,y_i)\in \RR^2$, $m\in \NN$, $(X,Y)\in (\mathbb{SR}^{2\times 2})^2$,
	$\lambda_i> 0$, $\xi> 0$ and $\sum_{i=1}^m \lambda_i+\xi=1$.
	Therefore 
		$$M:=\ \mc{M}_2-\xi \mc{M}^{(X,Y)}_2,$$
	is a cm moment matrix of rank at most 5 satisfying the relations
		$$\bY^2=\mds 1+\bX^2\quad\text{and}\quad \bX\bY=\bY\bX.$$
	By \cite{Fia14} and references therein, $M$ admits a measure if and only if $M$ is psd, RG and satisfies
	$\Rank M\leq\Card \mathcal V_M$.
	To conclude the proof it only 
	remains to prove that $X, Y$ are of the form (\ref{(X,Y)-form}). Note that $\mc{M}^{(X,Y)}_2$ is a nc moment matrix of
	rank 4. Therefore the columns
	$\{\mds 1,\bX,\bY,\bX\bY\}$ are linearly independent \cite[Corollary 2.3]{BZ18} and hence
		\begin{equation*}
			\bX^2=a_1 \mds 1+ b_1 \bX+ c_1 \bY+d_1 \bX\bY,\quad\text{and}\quad
			\bY^2=a_3 \mds1+ b_3 \bX+ c_3 \bY+d_3 \bX\bY,
		\end{equation*}
	where $a_j,b_j,c_j,d_j\in \RR$ for $j=1, 3$. By \cite[Theorem 3.1 (1)]{BZ18},
	$d_1=d_3=0$. 
	By \cite[Theorem 3.1 (3)]{BZ18},
	$c_1=b_3=0$. Since $\bY^2=\mds 1+\bX^2$ it follows that $b_1=c_3=0$ and $a_3=1+a_1$.
	By  \cite[Theorem 3.1 (4)]{BZ18}, $X$ and $Y$ are of the form 
	 (\ref{(X,Y)-form}).
	
	To prove the result about the type of the measure note that if a cm moment matrix which admits a measure 
	satisfies $\bY^2=\mds 1+\bX^2$, then it admits a measure with at most 5 atoms by
	the results of Curto and Fialkow \cite{CF98-1}, \cite{CF02}, \cite{Fia14} (see also \cite[Theorem 2.7]{BZ18}). 
	On the other hand there must be at least 2 cm atoms in every measure of type $(m,1)$, $m\in \NN$, for $\cM_2$, 
	otherwise $\cM_2$ would be of rank at most 5.
\end{proof} 
%
\begin{proof}[Proof of Theorem \ref{M(2)-bc3-r6-new1-cor}]
	Let us first prove the implication $(\Rightarrow)$. Suppose that $\beta$ admits a measure. 
	By Theorem \ref{M(2)-bc3-r6-new1}, $\mc M_2$ is of the form
		\begin{equation}\label{with-2-times-2-matrices-bc1}
			\mc{M}_2=\sum_{i=1}^m\lambda_i \mc{M}^{(x_i,y_i)}_2+
			\xi\mc{M}^{(X,Y)}_2,
		\end{equation}
	where $m\in \NN$, $(x_i,y_i)\in \RR^2$, $(X,Y)\in (\mathbb{SR}^{2\times 2})^2$,
	$\lambda_i> 0$, $\xi>0$ and $\sum_{i=1}^m \lambda_i+\xi=1$.
	By the form \eqref{(X,Y)-form} of $(X,Y)$ it is easy to check that 
		\begin{equation} \label{r6-form-of-x2-2-bc4-new1-cor} 
			\beta^{(X,Y)}_X=\beta^{(X,Y)}_Y=\beta^{(X,Y)}_{X^3}=
			\beta^{(X,Y)}_{X^2Y}=\beta^{(X,Y)}_{XY^2}=\beta^{(X,Y)}_{Y^3}=0,
		\end{equation}
	where $\beta^{(X,Y)}_{w}$ are the moments of $\mc{M}^{(X,Y)}_2$.
	Using (\ref{with-2-times-2-matrices-bc1}) and (\ref{r6-form-of-x2-2-bc4-new1-cor}), we conclude that
	$\sum_{i=1}^m \lambda_i \mc{M}^{(x_i,y_i)}_2$ and $\xi \mc{M}^{(X,Y)}_2$ are of the forms
			\begin{align}
				\begin{mpmatrix}
					a & \beta_X & \beta_Y & b & c & c &  a+b		\\
					\beta_X & b & c & \beta_{X^3} & \beta_{X^2Y} & \beta_{X^2Y} & \beta_X+\beta_{X^3} 					\\
					\beta_{Y}& c & a+b & \beta_{X^2Y} & \beta_X+\beta_{X^3} & \beta_X+\beta_{X^3} & \beta_Y+\beta_{X^2Y}			\\
					b & \beta_{X^3} & \beta_{X^2Y} & d & e & e & b+d	\\
					c & \beta_{X^2Y} & \beta_X+\beta_{X^3} & e & b+d & b+d & c+e					\\
					c & \beta_{X^2Y} & \beta_X+ \beta_{X^3}& e & b+d & b+d & c+e 					\\
					a+b & \beta_X+\beta_{X^3} & \beta_{Y}+\beta_{X^2Y} & b+d & c+e & c+e & a+2b+d
				\end{mpmatrix}, \label{matrix-1-bc1}\\
				\begin{mpmatrix}
					1-a & 0 & 0 & \beta_{X^2}-b & A_1(c) & A_1(c) &  A_2(a,b)	\\
					0 & \beta_{X^2}-b & A_1(c) & 0 & 0 & 0 & 0 					\\
					0 & A_1(c) & A_2(a,b) & 0 & 0 & 0 & 0			\\
					\beta_{X^2}-b & 0 & 0 & \beta_{X^4}-d & A_3(e) & A_3(e) & A_4(b,d)\\
					A_1(c) & 0 & 0 & A_3(e) & A_4(b,d) & \beta_{XYXY}-(b-d)& 					A_5(c,e)\\
					A_1(c) & 0 & 0 & A_3(e) & \beta_{XYXY}-(b-d) & A_4(b,d) & A_5(c,e)					\\
					A_2(a,b) & 0 & 0 & A_4(b,d) & A_5(c,e) &  A_5(c,e)  & A_6(a,b,d)
				\end{mpmatrix},\label{matrix-2-bc1}
			\end{align}
	where 
	\begin{equation*}
  	\begin{split}
		A_1(c)&= \beta_{XY}-c,\\
		A_3(e) &= \beta_{X^3Y}-e,\\
		A_5(c,e)&= \beta_{XY}+\beta_{X^3Y}-(c+e),
	\end{split}
	\qquad
	\begin{split}
		A_2(a,b)    &= 1+\beta_{X^2}-(a+b),\\ 
		A_4(b,d)    &=\beta_{X^2}+\beta_{X^4}-(b+d),\\
		A_6(a,b,d) &=1+2\beta_{X^2}+\beta_{X^4}-(a+2b+d),
	\end{split}
	\end{equation*}
	for some $a,b,c,d,e\in \RR$, 
	and observe that the matrix (\ref{matrix-1-bc1}) is $L(a,b,c,d,e)$
	and (\ref{matrix-2-bc1}) is $\mc{M}_2-L(a,b,c,d,e)$.
	Since $L(a,b,c,d,e)$ is a cm moment matrix which admits a measure, 
	conditions (\ref{point1-bc1-r6}) and (\ref{point4-bc1-r6}) of Theorem \ref{M(2)-bc3-r6-new1-cor} follow from
	\cite{Fia14} and references therein. Since $\mc{M}_2-L(a,b,c,d,e)$ is a nc moment matrix which admits a measure,
	(\ref{point2-bc1-r6}) and 
	(\ref{point3-bc1-r6}) of Theorem \ref{M(2)-bc3-r6-new1-cor} are true by Proposition \ref{Mn-psd} and 
	Corollary \ref{lin-ind-of-4-col} above. This proves the implication $(\Rightarrow)$.

	It remains to prove the implication $(\Leftarrow)$. We have to prove that conditions 
	(\ref{point1-bc1-r6})-(\ref{point4-bc1-r6}) imply that there is a measure for $\mc M_2$.
	Since $L(a,b,c,d,e)$ is a cm moment matrix that satisfies 
	(\ref{point1-bc1-r6}) and (\ref{point4-bc1-r6}), 
	it admits a measure by \cite{Fia14} and references therein. 
	Now note that  $M:=\mc{M}_2-L(a,b,c,d,e)$ is a nc moment matrix of the form \eqref{matrix-2-bc1} satisfying
		\begin{equation}\label{zero-moments-M} 
			\beta^{M}_X=\beta^{M}_Y=\beta^{M}_{X^3}=
			\beta^{M}_{X^2Y}=\beta^{M}_{XY^2}=\beta^{M}_{Y^3}=0,
		\end{equation}
	where $\beta^M_w$ denote the moments of $M$. 
	It remains to prove that $M$ admits a measure.
	By (\ref{point2-bc1-r6}), $M$ is psd, and from (\ref{point3-bc1-r6}), $M$ is of rank at least 4 with
	linearly independent columns $\mds 1, \bX, \bY, \bX\bY$. Since $M$ satisfies the relation
	$\bY^2=\mds 1+\bX^2$, it can be of rank at most 6. We separate three possibilities.\\
 
	\noindent\textbf{Case 1:}  $\Rank M=4$. From the form of $M$, we see that it must additionally satisfy  
	  \begin{equation*}
	   	X^2=a_1\mds 1, \quad \text{and} \quad
	   	\bX\bY+\bY\bX=a_2\mds 1,		
	  \end{equation*}
	for some $a_1,a_2\in \RR.$ Since
		$M$ is also psd, there exist a measure for $\beta$ by 
	\cite[Theorem 3.1 (3)]{BZ18}.\\
\\
	\noindent\textbf{Case 2:} $\Rank M=5$. By the form of $M$ and \eqref{point2-bc1-r6}, we have the additional relation
	  \begin{equation}\label{other-re}
	    a\mds 1+ d\bX^2+e(\bX\bY+\bY\bX)=\mbf{0}
	  \end{equation}
	for some $a,d,e\in \RR.$
	Since $M$ is psd and RG (since there are only quadratic column relations), 
	Lemma \ref{linear-trans} states that there is a transformation of the form \eqref{lin-trans1} which we may apply to get
	a moment sequence $\widetilde{\beta}_w$ with a moment matrix $\widetilde M$ satisfying the relations \eqref{first-rel} and \eqref{second-rel}. 
	By Lemma \ref{linear-trans-on-zero-moments} we have that 
		$$\widetilde{\beta}_{X}=
		\widetilde{\beta}_Y=
		\widetilde{\beta}_{X^3}=
		\widetilde{\beta}_{X^2Y}=
		\widetilde{\beta}_{Y^3}=0.$$
	Hence the measure for $\widetilde{\beta}_w$ exists by \cite[Theorems 6.5, 6.8, 6.11, 6.14]{BZ18}.\\
	\\
	\noindent\textbf{Case 3:} $\Rank M=6$.
	Since $M$ is psd, RG (since the only relation is $\bY^2=\mds 1+\bX^2$) and
	satisfies \eqref{zero-moments-M}, it admits a measure by Theorem \ref{M(2)-bc3-r6-new1} \eqref{M(2)-c3-r6-pt1}.\\
	
	The type of representing measure, as well as the sufficiency of \eqref{suff-cond-meas} can be inferred from Theorem \ref{M(2)-bc3-r6-new1}.
\end{proof}

Theorem \ref{M(2)-bc3-r6-new1-cor} (along with the others from \cite{BZ18}) provides with a new computational method for testing the existence of a measure. While searching for a flat extension from $\mc{M}_{2}$ to $\mc{M}_{3}$ is reasonable, this approach quickly becomes intractable if $\mc{M}_{2}$ admits positive extensions $\mc{M}_{k}$, for a large $k$, which then admits a flat extension to $\mc{M}_{k+1}$. Comparatively, checking the LMI's from Theorem \ref{M(2)-bc3-r6-new1-cor} always maintains the same level of computational complexity.
In the following example we present two psd moment matrices $\mc{M}_2$ satisfying $\bY^2=\mds 1+\bX^2$,
one which admits a representing measure and the other which does not. The proof is by the use of Theorem \ref{M(2)-bc3-r6-new1-cor}, with the computations easily checked in \textit{Mathematica}.

\begin{example}
For the moment matrix 
	$$\mc{M}_2=
\begin{pmatrix}
 1 & 0 & 0 & \frac{1}{2} & 0 & 0 & \frac{3}{2} \\
 0 & \frac{1}{2} & 0 & 0 & 0 & 0 & 0 \\
 0 & 0 & \frac{3}{2} & 0 & 0 & 0 & 0 \\
 \frac{1}{2} & 0 & 0 & 1 & 0 & 0 & \frac{3}{2} \\
 0 & 0 & 0 & 0 & \frac{3}{2} & 0 & 0 \\
 0 & 0 & 0 & 0 & 0 & \frac{3}{2} & 0 \\
 \frac{3}{2} & 0 & 0 & \frac{3}{2} & 0 & 0 & 3 \\
\end{pmatrix}$$
we proved in \cite[Example 8.16]{BZ18} that it admits a representing measure (but not a flat extension). We will check this fact also by the use of
Theorem \ref{M(2)-bc3-r6-new1-cor}. Using \textit{Mathematica} we get $a=0.75$, $b=c=d=e=0$ as a feasible solution of 
both LMI's from \eqref{point1-bc1-r6} and \eqref{point2-bc1-r6}. We check that the condition \eqref{point3-bc1-r6} of Theorem 
\ref{M(2)-bc3-r6-new1-cor} is also met, i.e., the eigenvalues are $1.5, 0.75, 0.5, 0.25$. The moment matrix $L(0.75,0,0,0,0)$
satisfies $\bX=\bX^2=\bX\bY=\bY\bX=\mbf 0$ and $\bY^2=\mds 1$, hence it is of rank 2. The corresponding variety is $\{(0,1),(0,-1)\}$,
so also the condition \eqref{point4-bc1-r6} of Theorem 
\ref{M(2)-bc3-r6-new1-cor} is satisfied. Thus $M$ indeed admits a measure by  Theorem 
\ref{M(2)-bc3-r6-new1-cor}.

For the moment matrix 
$$\mc{M}_2=
\left(
\begin{array}{ccccccc}
 1 & 0 & \frac{4}{15} & \frac{2}{3} & -\frac{32}{33} & -\frac{32}{33} & \frac{5}{3} \\[1.5mm]
 0 & \frac{2}{3} & -\frac{32}{33} & \frac{4}{15} & -\frac{8}{27} & -\frac{8}{27} & \frac{4}{15} \\[1.5mm]
 \frac{4}{15} & -\frac{32}{33} & \frac{5}{3} & -\frac{8}{27} & \frac{4}{15} & \frac{4}{15} & -\frac{4}{135} \\[1.5mm]
 \frac{2}{3} & \frac{4}{15} & -\frac{8}{27} & \frac{2}{3} & -\frac{8}{9} & -\frac{8}{9} & \frac{4}{3} \\[1.5mm]
 -\frac{32}{33} & -\frac{8}{27} & \frac{4}{15} & -\frac{8}{9} & \frac{4}{3} & \frac{10}{9} & -\frac{184}{99} \\[1.5mm]
 -\frac{32}{33} & -\frac{8}{27} & \frac{4}{15} & -\frac{8}{9} & \frac{10}{9} & \frac{4}{3} & -\frac{184}{99} \\[1.5mm]
 \frac{5}{3} & \frac{4}{15} & -\frac{4}{135} & \frac{4}{3} & -\frac{184}{99} & -\frac{184}{99} & 3 \\
\end{array}\right)$$
we check with \textit{Mathematica} that the eigenvalues are nonnegative, i.e., 
	$6.92,2.35,0.22,0.11,0.039,0.014,0.$
Clearly we have that $\bY^2=\mds{1}+\bX^2$. Using \textit{Mathematica} we check that the LMI's from Theorem \ref{M(2)-bc3-r6-new1-cor} \eqref{point1-bc1-r6}, \eqref{point2-bc1-r6} are not simultaneously feasible. Hence $\mc{M}_2$ does not admit a representing measure.
\end{example}


\section{$\mc M_2$ of rank 6 with relation $\bY^2=\mds 1$}
\label{S2}

The main result of this section, see Theorem \ref{main-res-2} below, 
is that moment matrices $\mc M_2$ generated by the 
atoms $(X,Y)$ of size 3 satisfying $Y^2=I_3$ can always be represented with atoms of size at most 2. Moreover,  if we consider a single atom of size 3, then a single atom of size 2 suffices. 

  \begin{theorem}\label{main-res-2}
    Let $\beta$ be a moment sequence with a nc moment matrix $\cM_2$ satisfying the column relation $\bY^2=\mds 1$.
    Then the following are equivalent:
    \begin{enumerate}
	\item\label{pt1} $\mc M_2$ admits a measure of type $(m_1,m_2,m_3)$, $m_1, m_2,m_3\in \NN\cup\{0\}$.
	\item\label{pt2} $\mc M_2$ admits a measure of type $(m_1,m_2)$, $m_1, m_2\in \NN \cup\{0\}$.
    \end{enumerate}
	Moreover, if $m_3=1$ in \eqref{pt1} the $m_2=1$ in \eqref{pt2}.
  \end{theorem} 

The proof is constructive and can be seen as the first step toward proving the following conjecture:

\begin{conjecture}
    Let $\beta$ be a moment sequence with a moment matrix $\cM_2$ satisfying the column relation $\bY^2=\mds 1$.
    Then the following are equivalent:
  \begin{enumerate}
	\item $\mc M_2$ admits a measure.
	\item $\mc M_2$ admits a measure of type $(m_1,m_2)$, $m_1, m_2\in \NN$.
	\item $\mc M_2$ admits a measure of type $(m,1)$, $m\in \NN$.
  \end{enumerate}
\end{conjecture}

Let $\beta^{(4)}$ be a truncated moment sequence and $\mc M_2$ its moment matrix. The notations 
	$\Delta(\beta^{(4)})$ and $\Delta(\mc M_2)$ will both denote the difference
	$$\Delta(\beta^{(4)})=\Delta(\mc M_2):=\beta_{X^2Y^2}-\beta_{XYXY},$$
which will be important in the analysis below.

To prove Theorem \ref{main-res-2} we first have to understand the form of moment matrices 
$\mc M_2^{(X,Y)}$ with $(X,Y)\in (\SSS\RR^{2\times 2})^2$ and $Y^2=I_2$. We illustrate this in the next lemma.


\begin{lemma}\label{2-by-2-atom}
	Let $(X,Y)\in (\SSS\RR^{2\times 2})^2$ be a pair of symmetric matrices of size 2 with $Y^2=I_2$ 
	and $\Delta(\mc M_2^{(X,Y)})\neq 0$. Then there is 
	$\widetilde X:=\begin{mpmatrix} a & b \\ b & c\end{mpmatrix}\in \SSS\RR^{2\times 2}$,
	such that
	  \begin{equation}\label{diag-Y}
	    \mc M_2^{(X,Y)}=\mc M_2^{(\widetilde X,\widetilde Y)},
	  \end{equation}
	where $\widetilde Y=\begin{mpmatrix} 1 & 0 \\ 0& -1\end{mpmatrix}.$
	Moreover, $\mc M_2^{(\widetilde X,\widetilde Y)}$ is equal to 
	$$\begin{mpmatrix}
 1 & \frac{1}{2} (a+c) & 0 & \frac{1}{2} \left(a^2+2 b^2+c^2\right) & \frac{1}{2} (a-c) & \frac{1}{2} (a-c) & 1 \\
 \frac{1}{2} (a+c) & C_4(a,b,c) & \frac{1}{2} (a-c) & C_3(a,b,c) & \frac{1}{2} (a-c) (a+c) & \frac{1}{2} (a-c) (a+c) &
   \frac{1}{2} (a+c) \\
 0 & \frac{1}{2} (a-c) & 1 & \frac{1}{2} (a-c) (a+c) & \frac{1}{2} (a+c) & \frac{1}{2} (a+c) & 0 \\
 C_4(a,b,c) &C_3(a,b,c) & \frac{1}{2} (a-c) (a+c) & C_1(a,b,c)&
   C_2(a,b,c) & C_2(a,b,c) & C_4(a,b,c) \\
 \frac{1}{2} (a-c) & \frac{1}{2} (a-c) (a+c) & \frac{1}{2} (a+c) & C_2(a,b,c) & C_4(a,b,c) & 
	C_5(a,b,c) & \frac{1}{2} (a-c) \\
 \frac{1}{2} (a-c) & \frac{1}{2} (a-c) (a+c) & \frac{1}{2} (a+c) & C_2(a,b,c) & C_5(a,b,c) & 
	C_4(a,b,c) & \frac{1}{2} (a-c) \\
 1 & \frac{1}{2} (a+c) & 0 & C_4(a,b,c) & \frac{1}{2} (a-c) & \frac{1}{2} (a-c) & 1
	\end{mpmatrix},$$
	where
	\begin{align*} 
		2C_1(a,b,c) &= 
		a^4+4 a^2 b^2+4 a b^2 c+2 b^4+4 b^2 c^2+c^4,\\
		2C_2(a,b,c) &= 
		(a-c) \left(a^2+a c+b^2+c^2\right),\\
		2C_3(a,b,c) &= 
		a^3+3 a b^2+3 b^2 c+c^3,\\
		2C_4(a,b,c) &= 
		a^2+2 b^2+c^2,\\
		2C_5(a,b,c) &= 
		a^2-2 b^2+c^2.\\
	\end{align*}
	In particular, we have that $\Delta(\mc M_2^{(\widetilde X,\widetilde Y)})=2b^2.$
\end{lemma}

\begin{proof}
	To prove \eqref{diag-Y} note that since $Y^2=I_2$ the eigenvalues of $Y$ are $1$ or $-1$. Since $\Delta(\mc M_2^{(X,Y)})
		\neq 0$, $X$ and $Y$
	do not commute. Hence there is an orthogonal matrix $U\in \RR^{2\times 2}$ such
	that $UYU^t=\begin{mpmatrix} 1 & 0 \\ 0& -1\end{mpmatrix}.$ Taking $\widetilde X=UXU^t$ proves \eqref{diag-Y}.
	The remaining part of the lemma can be easily checked.
\end{proof}

We will prove that for every pair $(X,Y)\in (\mathbb{SR}^{3\times 3})^2$ satisfying $Y^2=1$ we can write
	\begin{equation} \label{momY2=1}
		\mc{M}_2^{(X,Y)}=\sum_{i=1}^m \lambda_i \mc{M}^{(x_i,y_i)}_2 +  t \mc{M}^{(\widetilde X,\widetilde Y)}_2,
	\end{equation}
	where $(x_i,y_i)\in \RR^2$, $m\in \NN$, $(\widetilde X,\widetilde Y)\in (\mathbb{SR}^{2\times 2})^2$ as in 
	Lemma \ref{2-by-2-atom}, $\lambda_i> 0$, $t> 0$ and $\sum_{i=1}^m \lambda_i+t=1$.
	Since $\Delta(\mc{M}^{(x,y)}_2)=0$ for every $(x,y)\in \RR^2$ we must have
		$$\Delta :=\Delta(\mc{M}^{(X,Y)}_2)=t\cdot \Delta(\mc{M}^{(\widetilde X,\widetilde Y)}_2)=t\cdot 2b^2,$$
	where we used Lemma \ref{2-by-2-atom} for the second equality.
	Hence a decomposition of the form \eqref{momY2=1} requires that  $b=\sqrt{\frac{\Delta}{2t}}$ \big(we may WLOG assume $b$ is positive, since only even powers of $b$ appear in $\mc{M}^{(\widetilde X,\widetilde Y)}_2$\big). Notice that if $\Delta=0$, then we are in the commutative setting. So we may assume that $\Delta>0$.

\begin{lemma}\label{2-by-2-atom-new-info} 
	Let $(\widetilde X,\widetilde Y)\in (\mathbb{SR}^{2\times 2})^2$ as in Lemma \ref{2-by-2-atom}, with
	$b=\sqrt{\frac{\Delta}{2t}}$ for some $t>0$.
	We have that
	  $$t\cdot\cM_2^{(\widetilde X,\widetilde Y)}=B_1+B_2\cdot t+B_3\cdot \frac{1}{t},$$
	where
\begin{align*}
	B_1
	&=
\begin{mpmatrix}
 0 & 0 & 0 & \frac{1}{2}\Delta& 0 & 0 & 0 \\
 0 &\frac{1}{2}\Delta & 0 & \frac{3(a+c) }{4}\Delta  & 0 & 0 & 0 \\
 0 & 0 & 0 & 0 & 0 & 0 & 0 \\
\frac{1}{2}\Delta & \frac{3(a+c) }{4}\Delta & 0 &\left(a^2+ac+c^2\right) \Delta&\frac{(a-c) }{4}\Delta& \frac{(a-c) }{4}\Delta&\frac{1}{2}\Delta\\
 0 & 0 & 0 & \frac{(a-c) }{4}\Delta & \frac{1}{2}\Delta & -\frac{1}{2}\Delta & 0 \\
 0 & 0 & 0 & \frac{(a-c) }{4}\Delta & -\frac{1}{2}\Delta & \frac{1}{2}\Delta & 0 \\
 0 & 0 & 0 & \frac{1}{2}\Delta & 0 & 0 & 0 
\end{mpmatrix},\\
	B_2
	&=
\begin{mpmatrix}
1 & \frac{1}{2} (a+c) & 0 & C_{4,2}(a,c) & \frac{1}{2} (a-c) & \frac{1}{2} (a-c) & 1\\
 \frac{1}{2} (a+c) & C_{4,2}(a,c) & \frac{1}{2} (a-c) &C_{3,2}(a,c) & \frac{1}{2} (a-c) (a+c) & \frac{1}{2} (a-c) (a+c) & \frac{1}{2} 
   (a+c) \\
 0 & \frac{1}{2} (a-c) & 1 & \frac{1}{2} (a-c) (a+c) & \frac{1}{2} (a+c) & \frac{1}{2} (a+c) & 0 \\
 C_{4,2}(a,c) & C_{3,2}(a,c)  & \frac{1}{2} (a-c) (a+c) & C_{1,2}(a,c) & C_{2,2}(a,c)
   &C_{2,2}(a,c) & C_{4,2}(a,c)  \\
 \frac{1}{2} (a-c) & \frac{1}{2} (a-c) (a+c) & \frac{1}{2} (a+c) & C_{2,2}(a,c) & C_{4,2}(a,c)& C_{4,2}(a,c) &
   \frac{1}{2} (a-c) \\
 \frac{1}{2} (a-c) & \frac{1}{2} (a-c) (a+c) & \frac{1}{2} (a+c) & C_{2,2}(a,c) & C_{4,2}(a,c) & C_{4,2}(a,c) &
   \frac{1}{2} (a-c) \\
 1 & \frac{1}{2} (a+c) & 0 & C_{4,2}(a,c) & \frac{1}{2} (a-c) & \frac{1}{2} (a-c) & 1 \\
\end{mpmatrix},\\
	B_3
	&=
	\frac{\Delta^2}{4}\cdot E_{44},
\end{align*}
with
	\begin{equation*}
	\begin{split} 
		C_{1,2}(a,c) &= \frac{1}{2} \left( a^4 + c^4 \right) ,\\
		C_{2,2}(a,c) &=  \frac{1}{2} (a-c) \left(a^2+a c+c^2\right),
	\end{split}
	\qquad
	\begin{split}
		C_{3,2}(a,c) &=  \frac{1}{2} (a+c) \left(a^2-a c+c^2\right),\\
		C_{4,2}(a,c) &=  \frac{1}{2} (a^2+c^2),
	\end{split}
	\end{equation*}
and $E_{44}$ is the standard $7\times 7$ coordinate matrix with the only non-trivial entry in position $(4,4)$ being 1. Moreover, $B_2$ and $B_3$ are positive semidefinite.
\end{lemma}

\begin{proof}
	The statements about the form of $t\cdot\cM_2^{(\widetilde X,\widetilde Y)}$ can be easily checked by direct
	computation. It is obvious that $B_3$ is psd. 
	It remains to prove the fact that $B_2$ is psd. We know that $t\cdot\cM_2^{(\widetilde X,\widetilde Y)}$ is psd for every $t>0$. 	
	If $B_2$ has a negative eigenvalue, then $t\cdot\cM_2^{(\widetilde X,\widetilde Y)}$ also has a negative eigenvalue for
	$t>0$ big enough. 
	(Note that $\displaystyle \lim_{t\to\infty} B_3\frac{1}{t}=\mbf 0$.)
\end{proof}

The next lemma describes the moments generated by a pait $(X,Y)\in(\SSS\RR^{n\times n})^2$ with $Y^2=I_n$ where the multiplicities of the eigenvalues $1$, $-1$ are
$n-1$, $1$, respectively.

\begin{lemma}\label{gen-lemma-one--1}
	Let $(X,Y)\in (\SSS\RR^{n\times n})^2$, $t\geq 2$, be a pair of symmetric matrices of size $n$ such that $Y^2=I_n$ 
	and the multiplicities of the eigenvalues $1$, $-1$ are $n-1$, $1$, respectively.
	Then: 
	\begin{enumerate}
	  \item\label{Y2=1-pt1} $\mc M_2^{(X,Y)}=\mc M_2^{(\widetilde X,\widetilde Y)}$
		with
		\begin{equation}\label{form-of-atoms}
		  \widetilde X=\left(\begin{array}{cc} D & x \\ x^t & \alpha\end{array}\right),\quad \widetilde Y=\left(\begin{array}{cc} I_{n-1}  & 0 \\ 0 & -1 \end{array}\right),
		\end{equation}
		where $D\in\SSS\RR^{(n-1)\times (n-1)}$ is a diagonal matrix, $x\in \RR^{n-1}$ a vector, $\alpha\in\RR$ a real number, and
	  	  $$\widetilde X=WXW^t, \quad \widetilde Y=WYW^t,$$
		for some orthogonal matrix $W\in \RR^{n\times n}$.
          \item\label{Y2=1-pt2} $\mc M_2^{(X,Y)}$ admits a measure of type $t$ if $\mc M_2^{(\widehat X,\widetilde Y)}$ admits a measure of type $t$ 
		where 	
		\begin{equation*}
		  \widehat X=\left(\begin{array}{ccc} D_0 \oplus 0 & x \\ x^t & 0\end{array}\right) := \al{-}\left(\frac{\alpha+d_{n}}{2}\right)I_n+\widetilde X+\left(\frac{\alpha-d_n}{2}\right)\widetilde Y,
		\end{equation*}
		$\widetilde X, \widetilde Y$ are as in \eqref{Y2=1-pt1}, $d_n$ is the $(n-1)$-th diagonal entry of $D$ from \eqref{form-of-atoms}
		and $D_0$ a diagonal matrix of size $n-2$.
%
	  \item\label{Y2=1-pt3} $\mc M_2^{(\widehat X,\widetilde Y)}$ with $\widehat X$ and $\widetilde Y$ as in \eqref{Y2=1-pt2} is equal to
		$$\begin{mpmatrix}
		  1 & \beta_X & \frac{n-2}{n} & \beta_{X^2} & \beta_X & \beta_X & 1 \\
		  \beta_X & \beta_{X^2} & \beta_X & \beta_{X^3} & \beta_{X^2Y} & \beta_{X^2Y} & \beta_X \\
		  \frac{n-2}{n}& \beta_X & 1 & \beta_{X^2Y} & \beta_X & \beta_X & \frac{n-2}{n} \\
		  \beta_{X^2} & \beta_{X^3} &\beta_{X^2Y}   & \beta_{X^4} & \beta_{X^3Y} & \beta_{X^3Y} & \beta_{X^2}  \\
		\beta_{X} & \beta_{X^2Y} & \beta_{X}  & \beta_{X^3Y} & \beta_{X^2}  & \beta_{XYXY} & \beta_X \\
  		 \beta_{X} & \beta_{X^2Y} & \beta_{X}  & \beta_{X^3Y} & \beta_{XYXY}  & \beta_{X^2} & \beta_X \\
		  1 & \beta_X & \frac{n-2}{n} & \beta_{X^2} & \beta_X & \beta_X & 1 
		\end{mpmatrix},$$
	where
	\begin{equation*} 
	\begin{split}
	  \beta_X &= \frac{1}{n}\tr(D_0),\\
	  \beta_{X^2} &= \frac{1}{n}(\tr(D_0^2)+2x^tx),\\
          \beta_{X^3} &= \frac{1}{n}(\tr(D_0^3)+3\tr(\hat{D}xx^t)),\\
	  \beta_{X^2Y} &= \frac{1}{n}\tr(D_0^2),
	\end{split}
	\qquad
	\begin{split}
	 \beta_{X^3Y} &= \frac{1}{n}(\tr(D_0^3)+\tr(\hat{D}xx^t)),\\
	  \beta_{XYXY} &= \frac{1}{n}(\tr(D_0^2)-2x^tx),\\
          \beta_{X^4} &= \frac{1}{n}(\tr(D_0^4)+4\tr(\hat{D}^2xx^t)+2(x^tx)^2),\\
		&
	\end{split}
	\end{equation*}
	with $\hat{D} =  D_0 \oplus 0$. In particular, we have that
	\begin{align}  \label{mom-rel-2}
	    \beta_{X^2Y} &= \frac{1}{2}\left(\beta_{X^2}+\beta_{XYXY}\right),\\
 	   \beta_{X^3Y} &= \beta_{X^3}-\frac{2}{n}x^t\hat{D}x. \label{mom-rel-3}
	  \end{align}
	\end{enumerate}
\end{lemma}

\begin{proof}
	First we prove \eqref{Y2=1-pt1}. There is an orthogonal matrix $U\in \RR^{n\times n}$ such that $UYU^t=:\widetilde Y$ is of the form as in \eqref{form-of-atoms}.
	Further on, there is an orthogonal matrix $V_0\in \RR^{(n-1)\times (n-1)}$ such that by defining $V:=\begin{mpmatrix}V_0 & 0\\ 0 & 1\end{mpmatrix}$, the matrix
	$VUXU^tV =:\widetilde X$ is of the form  \eqref{form-of-atoms}. Since we also have $V\widetilde YV^t=:\widetilde Y$, defining $W=VU$ establishes \eqref{Y2=1-pt1}.
 
	Now we prove  \eqref{Y2=1-pt2}. By applying a linear transformation 
	$\phi(x,y)=(a+x+c y,y)$, where $a=\frac{-d_n-\alpha}{2}$, $c=\frac{\alpha-d_n}{2}$ and $d_n$ is the $(n-1)$-th diagonal entry of $D$ from \eqref{form-of-atoms}
	to the sequence $\beta^{(4)}$, we get a sequence $\widetilde \beta^{(4)}$ with $\mc M_2^{(\widehat X, \widetilde Y)}$ where $\widehat X$ and $\widetilde Y$ are as stated in \eqref{Y2=1-pt2}.
	Since the type of a measure remains unchanged when applying an invertible affine linear transformation, this proves  \eqref{Y2=1-pt2}. 

 	Part \eqref{Y2=1-pt3} of the lemma follows by direct calculation.
	See Appendix \ref{calc-part3-app} for the details.
\end{proof}

\begin{lemma}\label{gen-lemma-one--2}
	Let $(X,Y)\in (\SSS\RR^{n\times n})^2$, $n\geq 2$, be a pair of symmetric matrices of size $n$ of the form
	\begin{equation}\label{form-of-atoms-2}
		X=\left(\begin{array}{cc} D & x \\ x^t & 0\end{array}\right)\in \SSS\RR^{n\times n},\quad 
		Y=\left(\begin{array}{cc} I_{n-1}  & 0 \\ 0 & -1 \end{array}\right)\in \SSS\RR^{n\times n} ,
	\end{equation}
	where $D\in\SSS\RR^{(n-1)\times (n-1)}$ is a diagonal matrix, $x\in \RR^{n-1}$ is a vector.
	Let $(\widetilde X,\widetilde Y)\in (\SSS\RR^{2\times 2})^2$ be a pair of symmetric matrices of size $2$ of the form
	\begin{equation*} 
		\widetilde X=\left(\begin{array}{cc} a & b \\ b & c\end{array}\right)\in \SSS\RR^{2\times 2},\quad 
	  	\widetilde Y=\left(\begin{array}{cc} 1 & 0 \\ 0 & -1 \end{array}\right)\in \SSS\RR^{2\times 2},
	\end{equation*}
	with $b=\sqrt{\frac{\Delta}{2t}}$, $\Delta:=\Delta(\mc M_2^{(X,Y)})$, $t>0$ and
	$B_1, B_2, B_3$ as in Lemma \ref{2-by-2-atom-new-info}.
	If $\mc M_2^{(X,Y)}-t\mc M_2^{(\widetilde X,\widetilde Y)}$ is positive semidefinite for some $t>0$, then 
		$$c=0 \quad\text{and}\quad a=\frac{4x^t D x}{n\Delta}.$$
\end{lemma}

\begin{proof}
We begin by analyzing the kernel of $\big[\mc M_2^{(X,Y)}-B_1\big]_{\{\mds 1, \bX,\bY,\bX\bY\}}$.\\

\noindent{\textbf{Claim 1.}} $v:=\left(\begin{array}{cccc} 0 & -1 & 0 & 1\end{array}\right)^T\in \ker\big[\mc M_2^{(X,Y)}-B_1\big]_{\{\mds 1, \bX,\bY,\bX\bY\}}$.\\

Using Lemmas \ref{2-by-2-atom-new-info} and \ref{gen-lemma-one--1} 
	$$\left( \mc M_2^{(X,Y)} - B_1\right)|_{\{\mds 1, \bX,\bY,\bX\bY\}}=
	  \left(\begin{array}{cccc}  \ast & \beta_X &\ast & \beta_{XY}\\ \ast & \frac{1}{2}\left(\beta_{X^2} + \beta_{XYXY}\right) & \ast & \beta_{X^2Y} \\ \ast & \beta_{XY} & \ast & \beta_X \\ 
		\ast & \beta_{X^2Y} & \ast &  \frac{1}{2}\left(\beta_{X^2} + \beta_{XYXY}\right) \end{array}\right).$$
Moreover, using \eqref{mom-rel-2} we see that the second and the forth column of the matrix $\left[ \mc M_2^{(X,Y)} - B_1\right]_{\{\mds 1, \bX,\bY,\bX\bY\}}$ are equal.
Hence the vector 
	$v$
is in the kernel of  $\left[ \mc M_2^{(X,Y)} - B_1\right]_{\{\mds 1, \bX,\bY,\bX\bY\}}$.\\

Since $B_2$ and $B_3$ are psd by Lemma \ref{2-by-2-atom-new-info}, Claim 1 implies that $v$ must be in the kernel of both 
$[B_2]_{\{\mds 1, \bX,\bY,\bX\bY\}}$ and $[B_3]_{\{\mds 1, \bX,\bY,\bX\bY\}}$
if $\mc M_2^{(X,Y)}-t\mc M_2^{(\widetilde X,\widetilde Y)}$ is psd for some $t>0$.
We have that $[B_3]_{\{\mds 1, \bX,\bY,\bX\bY\}}=\mbf 0_4$ so $v$ is indeed in its kernel, while $[B_2]_{\{\mds 1, \bX,\bY,\bX\bY\}}v$ is equal to
$$\begin{mpmatrix}
1 & \frac{1}{2} (a+c) & 0 & \frac{1}{2} (a-c) \\
 \frac{1}{2} (a+c) & \frac{1}{2} \left(a^2+c^2\right) & \frac{1}{2} (a-c) & \frac{1}{2} (a-c) (a+c) \\
 0 & \frac{1}{2} (a-c) & 1 & \frac{1}{2} (a+c) \\
 \frac{1}{2} (a-c) & \frac{1}{2} (a-c) (a+c) & \frac{1}{2} (a+c) & \frac{1}{2} \left(a^2+c^2\right)\end{mpmatrix} v=
c\cdot 
\begin{mpmatrix}
 -1\\ -c \\ 1\\c
\end{mpmatrix}.
$$
Hence we must have $c=0$.\\

\noindent{\textbf{Claim 2.}} If $\mc M_2^{(X,Y)}-t\mc M_2^{(\widetilde X,\widetilde Y)}$ is psd, then 
	$\widetilde v:=\left(\begin{array}{ccccc} 0 & -1 & 0 & 0 & 1\end{array}\right)^T\in \ker\big[\mc M_2^{(X,Y)}-B_1\big]_{\{\mds 1, \bX,\bY,\bX^2, \bX\bY\}}$.\\

By Claim 1 it easily follows that 
\begin{equation}\label{eq0}
	{\widetilde v}^T \big[\mc M_2^{(X,Y)}-B_1\big]_{\{\mds 1, \bX,\bY,\bX^2, \bX\bY\}}\widetilde v=0.
\end{equation} 
If $\mc M_2^{(X,Y)}-t\mc M_2^{(\widetilde X,\widetilde Y)}$
is psd, $\mc M_2^{(X,Y)}-B_1$ is psd and by \eqref{eq0} Claim 2 follows.\\

Using Claim 2 and $c=0$, $\mc M_2^{(X,Y)}-t\mc M_2^{(\widetilde X,\widetilde Y)}$ being psd for some $t>0$, implies that
	$$\beta_{X^3}-\frac{3\Delta}{4}a=\left[ \mc M_2^{(X,Y)} - B_1\right]_{\{\{\bX^2\},\{\bX\}\}}=\left[ \mc M_2^{(X,Y)} - B_1\right]_{\{\{\bX^2\},\{\bX\bY\}\}}=\beta_{X^3Y}-\frac{\Delta}{4}a,$$
which further implies that
	$$a=\frac{2}{\Delta}(\beta_{X^3}-\beta_{X^3Y})=\frac{4x^t D x}{n\Delta},$$
where we used \eqref{mom-rel-3} for the second equality. This proves the lemma.
\end{proof}

\begin{proof}[Proof of Theorem \ref{main-res-2}]
We have to prove that $\cM_2^{(X,Y)}$, where $(X,Y)\in (\SSS\RR^{3\times 3})^2$ and $Y^2=I_3$,
has a measure of type $(m_1,m_2)$, where $m_1,m_2\in \NN\cup\{0\}$.

If $Y$ has all eigenvalues equal to 1 or $-1$, then $X$ and $Y$ commute and there is an orthogonal transformation $U\in \RR^{3\times 3}$
such that $UXU^t$ is diagonal and $UYU^t=\pm I_3$. Since $\cM_2^{(X,Y)}=\cM_2^{(UXU^t,UYU^t)}$, there exists a measure consisting of $m_1\leq 3$, atoms of size 1. 

Else $Y$ has two eigenvalues of the same sign and the third of the other. We may assume WLOG that two eigenvalues are 1 and the third is $-1$
(otherwise we do an affine linear transformation $(x,y)\mapsto (x,-y)$).
By Lemma \ref{gen-lemma-one--1} \eqref{Y2=1-pt2} it is enough to prove that $\cM_2^{(X,Y)}$
has a measure of type $(m_1,m_2)$, where $m_1,m_2\in \NN\cup\{0\}$,
for 
  $$X=\begin{mpmatrix} x_1 & 0 & x_2\\ 0 & 0 & x_3 \\ x_2 & x_3 & 0\end{mpmatrix},\quad Y=\begin{mpmatrix} 1 & 0 & 0\\ 0& 1& 0\\ 0& 0&-1\end{mpmatrix},$$
where $x_1,x_2,x_3\in \RR$.
We will separate two cases.
\\
\noindent\textbf{Case 1.} $x_1=0$ or $x_2=0$ or $x_3=0$:\\

	If $x_1=0$, we have $XY+YX=0$ and $\mc M_2^{(X,Y)}$ is of rank at most 5. By \cite[Theorems 3.1, 6.5, 6.8, 6.11, 6.14]{BZ18} it follows that $\mc M_2^{(X,Y)}$ admits a measure of type $(m_1,1)$ where $m_1\in \NN$.
	
	If $x_2=0$, the subspace $\Span\{e_1\}$ is reducing for $X$ and $Y$, and we can replace $(X,Y)$ by $(x_1,1)$ of density $\frac{1}{3}$ and $\Big(\begin{mpmatrix}0 & x_3\\ x_3 & 0\end{mpmatrix}, \begin{mpmatrix}1 & 0\\ 0 & -1\end{mpmatrix}\Big)$
	of densitiy $\frac{2}{3}$. 
	
	If $x_3=0$, the subspace $\Span\{e_2\}$ is reducing for $X$ and $Y$, and we can replace $(X,Y)$ by $(0,1)$ of density $\frac{1}{3}$ and $\Big(\begin{mpmatrix}x_1 & x_2\\ x_2 & 0\end{mpmatrix}, \begin{mpmatrix}1 & 0\\ 0 & -1\end{mpmatrix}\Big)$
	 of density $\frac{2}{3}$.
	This proves the theorem in Case 1.\\

\noindent\textbf{Case 2.} $x_1\neq 0$ and $x_2\neq 0$ and $x_3\neq 0$:\\

We will prove that $\cM_2^{(X,Y)}$ admits a measure of type $(m_1,1)$, $m_1\in \NN$. We denote by $(X_1,Y_1)\in (\SSS\RR^{2\times 2})^2$ 
the atom of size 2 and by $t$ its density.
By Lemma \ref{2-by-2-atom} we may assume that 
  $X_1=\begin{mpmatrix} a & b \\ b & c\end{mpmatrix}$, $Y_1=\begin{mpmatrix} 1 & 0 \\ 0& -1\end{mpmatrix}.$
Furthermore, by Lemma \ref{2-by-2-atom} we must have 
	\begin{equation}\label{t-and-diff-of-nc-mom}
		b=\pm \sqrt{\frac{1}{2t}\Delta(\mc M_2^{(X,Y)})}=\pm \sqrt{\frac{2}{3t}(x_2^2+x_3^2)}.
	\end{equation} 
Since $\beta_{Y}(\cM_2^{(X,Y)})=\frac{1}{3}$, $\beta_{Y}(\cM_2^{(X_1,Y_1)})=0$ and $\beta_{Y}(\cM_2^{(x_i,y_i)})=\pm 1$ for every atom $(x_i,y_i)$ of size $1$,
the sum $\sum_{i} \mu_i$ of the densities $\mu_i$ of atoms of size 1 must be at least $\frac{1}{3}$. Hence, the density $t$ satisfies $t\leq \frac{2}{3}$. Since the atoms of size 1 are not
sufficient, we have that $t>0$. To prove the theorem in Case 2 it suffices to prove the following claim.\\

\noindent\textbf{Claim.} There exists $t\in (0,\frac{2}{3}]$ such that 
		$$F(t):=\mc M_2^{(X,Y)}-t\cdot \mc M_2^{(X_1,Y_1)}$$ 
	admits a measure consisting of $m_1\in \NN$ atoms of size 1.\\

The necessarry condition for $F(t)$, $t>0$, to admit a measure is $F(t)\succeq 0$.
By Lemma \ref{gen-lemma-one--2} 
we must have $c=0$ and $a=\frac{x_1x_2^2}{x_2^2+x_3^2}$ in $X_{1}$.
Let $B_1,B_2,B_3$ be as in Lemma \ref{2-by-2-atom-new-info}. 
We have that 
\begin{align*}
  F(t)
  &=\mc M_2^{(X,Y)}-B_1-tB_2-\frac{1}{t}B_3\\
  &=\underbrace{\frac{1}{3}\begin{mpmatrix}
		 3 &  x_1 & 1 & x_1^2  & x_1 & x_1 & 3\\
		 x_1 & x_1^2  & x_1 & x_1^3 & x_1^2 & x_1^2 & x_1 \\
		  1& x_1 & 3 &  x_1^2 &  x_1^2 &  x_1 & 1 \\
		 x_1^2  & x_1^3 & x_1^2   & C(x_1,x_2,x_3) & x_1^3 & x_1^3&  x_1^2  \\
		x_1 &  x_1^2 &  x_1 &  x_1^3 &  x_1^2   &  x_1^2  &  x_1 \\
  		x_1 &  x_1^2 &  x_1 &  x_1^3 &  x_1^2   &  x_1^2 &  x_1 \\
		  3 &  x_1 & 1 &  x_1^2   &x_1&x_1 & 3
  \end{mpmatrix}}_{\mc M_2^{(X,Y)}-B_1}-
 \underbrace{\frac{t}{2}\begin{mpmatrix}
	2 & a & 0 & a^2 & a & a & 2\\
	a & a^2  & a  & a^3 & a^2 & a^2 & a \\
	 0 & a & 2 & a^2 & a & a & 0 \\
	a^2 & a^3  & a^2 & a^4 & a^3 &a^3 & a^2  \\
	a & a^2 & a & a^3  & a^2 & a^2  & a \\
	a & a^2 & a & a^3  & a^2 & a^2 & a \\
	2 & a & 0 & a^2 & a & a & 2 \\
	\end{mpmatrix}}_{tB_2}-\underbrace{\frac{1}{9t}(x_2^2+x_3^2)^2E_{44}}_{\frac{1}{t}B_3},
\end{align*}
where
	$$C(x_1,x_2,x_3) =x_1^4+4\frac{x_1^2x_2^2x_3^2}{x_2^2+x_3^2}+2(x_2^2+x_3^2)^2,$$
and the forms of $\mc M_2^{(X,Y)}$, $B_1, B_2, B_3$ are from Lemmas \ref{gen-lemma-one--1} \eqref{Y2=1-pt3}, \ref{gen-lemma-one--2}.
Clearly the kernels of $\mc M_2^{(X,Y)}-B_1$, $B_2$ and $B_3$ contain the vectors
	\begin{equation*}
		v_1 = (-1,0,0,0,0,0,1)^T,\quad v_2 = (0,-1,0,0,0,1,0)^T,\quad v_3 = (0,-1,0,0,1,0,0)^T.
	\end{equation*}
Hence, to prove that $F(t)$ is psd for some $t>0$ it is enough to consider the submatrix
	 $$[F(t)]_{\{\mds 1,\bX,\bY,\bX^2\}}.$$
Its principal minors are the following
	\begin{align*}
	  \det\left([F(t)]_{\{\mds 1\}}\right) 
	  &= 1-t,\\
	\det\left([F(t)]_{\{\mds 1,\bX\}}\right)  
	  &= \frac{x_1^2 \left(\left(9 t^2-18 t+8\right) x_2^4+4 (4-3 t) x_2^2 x_3^2+4 (2-3 t) x_3^4\right)}{36 \left(x_2^2+x_3^2\right)^2},\\
	\det\left([F(t)]_{\{\mds 1,\bX,\bY\}}\right)  
	  &= \frac{(2-3 t) x_1^2 \left((2-3 t) x_2^4+(2-3 t) x_3^4+4 x_2^2 x_3^2\right)}{27 \left(x_2^2+x_3^2\right)^2},\\
	\det\left([F(t)]_{\{\mds 1,\bX,\bY,\bX^2\}}\right)
	  &=\frac{1}{243 t ((x_2^2+x_3^2)^4) x_1^2} f(t),\\
	\end{align*}
where
	$$f(t)=f_0(x_1,x_2,x_3)+t\cdot f_1(x_1,x_2,x_3)+t^2\cdot f_2(x_1,x_2,x_3)+t^3\cdot f_3(x_1,x_2,x_3),$$
and
	\begin{align*}
	  f_0(x_1,x_2,x_3)
	  &=-16 (x_2^2 + x_3^2)^6,\\
	  f_1(x_1,x_2,x_3)
	  &=24 (x_2^2 + x_3^2)^3 (3 x_2^6 + 7 x_2^4 x_3^2 + 
   3 x_3^6 + x_2^2 (2 x_1^2 x_3^2 + 7 x_3^4)),\\
	  f_2(x_1,x_2,x_3)
	  & =-18 (6 x_2^{12} + 
   28 x_2^{10} x_3^2 + 6 x_3^{12} + 4 x_2^2 x_3^8 (2 x_1^2 + 7 x_3^2) + 
   x_2^8 (8 x_1^2 x_3^2 + 58 x_3^4) +\\ 
   	&\hspace{2cm}+x_2^4 x_3^4 (x_1^4 + 16 x_1^2 x_3^2 + 58 x_3^4) + 
   8 x_2^6 (2 x_1^2 x_3^4 + 9 x_3^6)),\\
	  f_3(x_1,x_2,x_3)
	  &=27 (2 x_2^{12} + 8 x_2^{10} x_3^2 + 
   2 x_3^{12} + 4 x_2^2 x_3^8 (x_1^2 + 2 x_3^2) + 
   4 x_2^6 x_3^4 (x_1^2 + 4 x_3^2) + 2 x_2^8 (2 x_1^2 x_3^2 + 7 x_3^4) + \\
  	&\hspace{2cm}+ x_2^4 x_3^4 (x_1^4 + 4 x_1^2 x_3^2 + 14 x_3^4)).	
	\end{align*}
For $t=\frac{2}{3}$ we get
	\begin{equation*}
	\begin{split}
	  \det\left([F\big(\frac{2}{3}\big)]_{\{\mds 1\}}\right) 
	  &= \frac{1}{3},\\
	 \det\left([F\big(\frac{2}{3}\big)]_{\{\mds 1,\bX,\bY\}}\right)   
	  &=0, 
	\end{split}
	\qquad
	\begin{split}
	  \det\left([F\big(\frac{2}{3}\big)]_{\{\mds 1,\bX\}}\right) &=\frac{2 x_1^2 x_2^2 x_3^2}{9 \left(x_2^2+x_3^2\right)^2},\\
	  \det\left([F\big(\frac{2}{3}\big)]_{\{\mds 1,\bX,\bY,\bX^2\}}\right) &=0.
	\end{split}
	\end{equation*} 
In addition we also calculate
	\begin{equation}\label{1-X-X2-det}
	\det\left([F\big(\frac{2}{3}\big)]_{\{\mds 1,\bX,\bX^2\}}\right)=-\frac{x_1^4 x_2^4 x_3^4 (x_1^2-8 (x_2^2+x_3^2))}{27 (x_2^2+x_3^2)^4}.
	\end{equation}
According to \eqref{1-X-X2-det} there are two cases to consider.\\

\noindent \textbf{Case 2.1.} $x_1^2-8 (x_2^2+x_3^2)\leq 0$:\\

It is easy to check that the columns $\mds 1$ and $\bY$ of 
$F\big(\frac{2}{3}\big)$ are both equal to
	$$\Big(
\begin{array}{ccccccc}
 \frac{1}{3} & \frac{x_1 x_5^2}{3
   \left(x_3^2+x_5^2\right)} & \frac{1}{3} &
   \frac{x_1^2 x_5^2 \left(2
   x_3^2+x_5^2\right)}{3
   \left(x_3^2+x_5^2\right)^2} &
   \frac{x_1 x_5^2}{3
   \left(x_3^2+x_5^2\right)} &
   \frac{x_1 x_5^2}{3
   \left(x_3^2+x_5^2\right)} & \frac{1}{3}
\end{array}
\Big)^t.$$
Hence $F\big(\frac{2}{3}\big)$ satisfies the relations $\bY=\mds 1$, $\bX\bY=\bY\bX=\bX$, $\bY^2=\mds 1$. 
Since 
	$$\det\left([F\big(\frac{2}{3}\big)]_{\{\mds 1\}}\right) >0, \quad \det\left([F\big(\frac{2}{3}\big)]_{\{\mds 1,\bX\}}\right)  > 0\quad
	\text{and}\quad\det\left([F\big(\frac{2}{3}\big)]_{\{\mds 1,\bX,\bX^2\}}\right) \geq 0,$$ 
$[F\big(\frac{2}{3}\big)]_{\{\mds 1,\bX,\bX^2\}}$ is psd matrix of rank 2 or 3. 
Hence $F(\frac{2}{3})$ is a psd commutative moment matrix of rank 2 or 3.
If $x_1^2-8 (x_2^2+x_3^2)=0$ then the fifth relation is $\bX^2=a_0\mds 1+a_1\bX$ for some $a_0,a_1\in \RR$.
Thus, it is recursively generated and  by the results of Curto and Fialkow \cite{CF98-1}, \cite{CF02}, \cite{Fia14} (see also \cite[Theorem 2.7]{BZ18})  it admits a measure consisting of 2 or 3 commutative atoms.\\

\noindent \textbf{Case 2.2.} $x_1^2-8 (x_2^2+x_3^2)>0$:\\
 
It is easy to see that for $0<t<\frac{2}{3}$ we have that 
	\begin{equation*}
	  \det\left([F(t)]_{\{\mds 1\}}\right) >0, \quad  \det\left([F(t)]_{\{\mds 1,\bX\}}\right) >0\quad\text{and}\quad   \det\left([F(t)]_{\{\mds 1,\bX,\bY\}}\right)>0.
	\end{equation*}
Since $\det\left([F(\frac{2}{3})]_{\{\mds 1,\bX,\bY,\bX^2\}}\right)=0$, we have that $f(\frac{2}{3})=0$, and hence
	$$f(t)=\big(\frac{2}{3}-t\big)g(t)$$
for some polynomial $g(t)$ which is quadratic in $t$. The polynomial $g(t)$ has a negative leading coefficient which implies that $g(t)$ achieves its maximum at $t_0$
satisfying $g'(t_0)=0$.
A calculation reveals $t_0$ to be
$$\frac{4 \left(x_2^2+x_3^2\right)^3 \left(x_2^2 x_3^2 \left(x_1^2+2 x_3^2\right)+x_2^6+2 x_2^4 x_3^2+x_3^6\right)}{3 \left(2 x_2^8
   \left(2 x_1^2 x_3^2+7 x_3^4\right)+4 x_2^6 x_3^4 \left(x_1^2+4 x_3^2\right)+4 x_2^2 x_3^8 \left(x_1^2+2 x_3^2\right)+x_2^4
   x_3^4 \left(x_1^4+4 x_1^2 x_3^2+14 x_3^4\right)+2 x_2^{12}+8 x_2^{10} x_3^2+2 x_3^{12}\right)}.$$
Moreover, 
$g(t_0)$ equals 
$$\frac{24 x_2^4 x_3^4
   \left(x_2^2+x_3^2\right)^6 \left(x_1^4+4 x_1^2
   \left(x_2^2+x_3^2\right)+2
   \left(x_2^2+x_3^2\right)^2\right)}{2 x_2^8 \left(2
   x_1^2 x_3^2+7 x_3^4\right)+4 x_2^6 x_3^4
   \left(x_1^2+4 x_3^2\right)+4 x_2^2 x_3^8
   \left(x_1^2+2 x_3^2\right)+x_2^4 x_3^4
   \left(x_1^4+4 x_1^2 x_3^2+14 x_3^4\right)+2
   x_2^{12}+8 x_2^{10} x_3^2+2 x_3^{12}},$$
which is strictly positive as the numerator and denominator of $g(t_{0})$ are sum of squares, and $x_{i}\neq 0$.
Now we only need that $0<t_0<\frac{2}{3}$.
The numerator and the denominator of $t_0$ are linear combinations of monomials 
	$$ x_2^{12} ,\; x_2^{10}x_3^{2} ,\; x_2^{8}x_3^{4} ,\; x_2^{6}x_3^{6} ,\;x_2^{4}x_3^8 ,\; x_2^{2}x_3^{10} ,\; x_2^{12} ,\; x_1^2x_2^{8}x_3^{2} ,\;
		x_1^2x_2^{6}x_3^{4},\;x_1^2x_2^{4}x_3^{6},\; x_1^2x_2^{2}x_3^{8},\; x_1^4x_2^{4}x_3^{4},$$
with the following coefficients:

$$\begin{array}{c|c|c|c|c|c|c|c|c|c|c|c|c}
\text{monomial}&x_2^{12} &  x_2^{10}x_3^{2} &  x_2^{8}x_3^{4} &  x_2^{6}x_3^{6} & x_2^{4}x_3^8 &  x_2^{2}x_3^{10} &  x_2^{12} &  x_1^2x_2^{8}x_3^{2} &
x_1^2x_2^{6}x_3^{4}& x_1^2x_2^{4}x_3^{6}&  x_1^2x_2^{2}x_3^{8}&  x_1^4x_2^{4}x_3^{4} \\
\hline
\text{numerator}& 4 & 20 & 44 & 56 & 44 & 20 & 4 & 4 & 12 & 12 & 4 & 0\\ 
\hline
\text{denominator}& 6 & 24 & 42 & 48 & 42 & 24 & 6 & 12 & 12 & 12 & 12 & 3\\
\end{array}.$$
Since we are in Case 2.2 we can use the inequality 
	$$x_1^2>8x_{2}^2+8x_{3}^2$$ 
to estimate
  \begin{eqnarray}
	x_1^2x_2^{8}x_3^{2} &>& 8x_2^{10}x_3^2+8x_2^{8}x_3^4, \label{ineq1}\\
	x_1^2x_2^{2}x_3^{8} &>& 8x_2^{4}x_3^{8}+8x_2^2x_3^{10},\label{ineq2}\\
	x_1^4x_2^{4}x_3^{4} &>& 8x_1^2x_2^{6}x_3^4+8x_1^2x_2^{4}x_3^6 \label{ineq3}\\
	x_1^4x_2^{4}x_3^{4} &>& 64x_2^8x_3^4+128x_2^6x_3^6+64x_2^4x_3^8.\label{ineq4}
  \end{eqnarray}
Summing up all the inequalities \eqref{ineq1}-\eqref{ineq4} we see that 
\begin{equation*} 
  x_1^2x_2^{8}x_3^{2}+x_1^2x_2^{2}x_3^{8}+2x_1^4x_2^{4}x_3^{4} 
	> 8x_2^{10}x_3^2+72x_2^{8}x_3^4+8x_2^{2}x_3^{10}+72x_2^{4}x_3^8+8x_1^2x_2^{6}x_3^4+8x_1^2x_2^{4}x_3^6+128x_2^6x_3^6.
\end{equation*}
Using this inequality we estimate the denominator from below by the coefficients:
$$\begin{array}{c|c|c|c|c|c|c|c|c|c|c|c|c}
\text{monomial}&x_2^{12} &  x_2^{10}x_3^{2} &  x_2^{8}x_3^{4} &  x_2^{6}x_3^{6} & x_2^{4}x_3^8 &  x_2^{2}x_3^{10} &  x_2^{12} &  x_1^2x_2^{8}x_3^{2} &
x_1^2x_2^{6}x_3^{4}& x_1^2x_2^{4}x_3^{6}&  x_1^2x_2^{2}x_3^{8}&  x_1^4x_2^{4}x_3^{4} \\
\hline
\text{lower bound}& 6 & 32 & 114 & 176 & 114 & 32 & 6 & 11 & 20 & 20 & 11 & 1
\end{array}.$$
Since all the coefficients of the lower bound on the denominator are at least $\frac{3}{2}$ times the corresponding coefficients of the numerator with strict inequalities at some coefficients, we conclude that the denominator is
bigger that $\frac{3}{2}$ of the numerator and hence $t_0< \frac{2}{3}$. Hence $F(t_0)$ is a cm moment matrix of rank 4, which is RG and psd with the cm variety $\{(x,y)\colon y=1\} \cup\{(0,-1)\}$ of infinite cardinality. Hence it
admits a measure consisting of atoms of size 1 by the results of Curto and Fialkow see \cite{Fia14} and reference therein. This settles Case 2.2, and concludes the proof of the Claim. Thus the theorem is proved.
\end{proof}

\begin{remark}
\begin{enumerate}
\item Note that Lemma \ref{gen-lemma-one--2} is true for any $n$ not only $n=3$. Hence if $Y$ has only 1 eigenvalue of some sign, then the atom of size 2 is uniquely determined up to density. Numerical experiments show that even in this case Claim 2 from the proof of Theorem  \ref{main-res-2} is true, but we were not able to find a theoretical argument for this observation as in the case $n=3$. So in the future research we plan to find some argument for the existence of such $t$ without using brute force methods.
\item If $Y$ has multiplicity of both eigenvalues at least 2, then possible atoms of size 2 in the measure are not unique anymore (up to density), so
some other construction of the measure is needed.
\item The characterization of finite sequences of real numbers that are the moments of \textit{one-atomic} tracial measures is deeply connected with Horn's problem (cf., \cite{CW18}). One approach to solve Horn's problem for $n\in\mbb{N}$, is to instead solve the one-atomic bivariate tracial moment problem of degree $2n-2$. 
In particular, solving the bivariate quartic tracial moment problem with the restriction of representing measures having a single size 3 atom $(X,Y)\in (\SSS\RR^{3\times 3})^2$, solves Horn's problem for $n=3$. The results of \cite{BZ18} and the analysis of this section do precisely this in the singular case, i.e., when the moment matrix $\mc M_2^{(X,Y)}$ is singular.
\end{enumerate}
\end{remark}

\section{Extension to $\mc{M}_n$ with two relations in $\mc M_2$}
\label{S3}

The main result of this subsection, Theorem \ref{M_n-XY+YX=0-all-in-one} below, extends the results for the existence of the measure for $\cM_n$, 
with two quadratic column relations, from $n=2$ 
(see \cite[Theorems 6.5, 6.8, 6.11, 6.14]{BZ18}) to an arbitrary $n\in \NN$.

Throughout this section, unless otherwise stated we assume that $n\geq 2$. We will also frequently be considering $[\mc{M}_{n}]_{\{\mds{1}, \bX, \bY, \bX^{2}, \bX\bY, \bY\bX, \bY^{2}\}}$, the quadratic component of $\mc{M}_{n}$. Thus we introduce the notation
	$$
		\mc{M}_{Q} := [\mc{M}_{n}]_{\{\mds{1}, \bX, \bY, \bX^{2}, \bX\bY, \bY\bX, \bY^{2}\}}.
	$$
We say that $\mc{M}_{n}$ is in \textbf{canonical form}, if 
it satisfies the relation 
	$$\bX\bY+\bY\bX = \mbf{0}$$ 
and one of the following relations
\begin{equation}\label{second-column-relation} 
\bY^2=\mbf 1-\bX^2\quad{\text{or}}\quad \bY^2=\mds{1}
\quad{\text{or}}\quad \bY^2=\mds{1}+\bX^2\quad{\text{or}}\quad \bY^2=\bX^2.
\end{equation}

We begin by showing that every $\mc{M}_{n}$, with $\mc{M}_{Q}$ of rank 5, can be transformed into a canonical form.  
\begin{lemma}\label{possible-relations}
	Suppose $\beta\equiv \beta^{(2n)}$ is a 
	nc sequence with a moment matrix $\mc{M}_n$, such that $\mc{M}_{Q}$ is of rank 5. If $\mc M_n$ is positive semidefinite and recursively generated, then there
	exists an affine linear transformation $\phi$ such that
	the sequence $\widehat \beta$, given by $\widehat \beta=L_\beta(w\circ\phi)$ 
	has a moment matrix 
	$\widehat{\mc{M}_n}$ in a canonical form.
\end{lemma}

%
%
%

\begin{proof}[Proof of Lemma \ref{possible-relations}]
	By \cite[Proposition 4.1 (1)]{BZ18} there exists a transformation $\phi$ such that $\widehat{\mc{M}_{Q}}$ is in a
	canonical form (Note that the assumption of \cite[Proposition 4.1 (1)]{BZ18} that $\cM_2$ admits a measure can be replaced by $\cM_2$ is psd and RG since only these two properties are used in the proof.)
	Since $\mc M_n$  (and hence also $\widehat{\mc{M}_n}$) is psd, 
	we conclude by \cite[Proposition 3.9]{CF96} that the relations from $\widehat{\mc{M}_{Q}}$ must also hold in $\widehat{\mc{M}_n}$. This proves the lemma.
\end{proof}

\begin{theorem} \label{M_n-XY+YX=0-all-in-one}
	Suppose $\beta\equiv \beta^{(2n)}$ is a  
	nc sequence with a moment matrix $\mc{M}_n$, which is positive semidefinite,
	recursively generated and $\mc{M}_{Q}$ is of rank 5.
	Then $\beta$ admits a nc measure if and only if  in the canonical form, with $\widehat{\mc{M}_n}$ and $\widehat{\beta}_w$ we have  
		$$\widehat{\mc{M}_n}-|\widehat{\beta}_X| 
			\cM_{n}^{(\sign(\widehat{\beta}_X)1,0)}-
			|\widehat{\beta}_Y|\cM_n^{(0,\sign(\widehat{\beta}_Y)1)}$$
	is positive semidefinite and recursively generated.
	Moreover, all the atoms in the measure are of size at most 2. 
\end{theorem}

Given an $\mc{M}_{n}$ in canonical form the column space of $\mc{M}_{n}$ is easily described.
  
\begin{lemma}\label{2-to-n-extension-lemma}
	Suppose that $\cM_n$ is recursively generated and in a canonical form. 
	Then we have the following:
		\begin{enumerate}[(1)]
			\item $\mc{M}_n$ satisfies the relation $\bX^i\bY+(-1)^{i+1}\bY\bX^i=\mbf 0$ for every $i\in \{1,\dotsc,n-1\}$.
			\item The column space $\mathcal C_{\mathcal{M}_n}$ of  $\mathcal{M}_n$ is equal to
			$$\mathcal C_{\mathcal{M}_n}:=\Span{\Big( \{\mathds{1}\}
				\bigcup \bigcup_{i=1}^n \{ \bX^i, \bX^{i-1}\bY \}\Big)}.$$
		\end{enumerate}
\end{lemma}

\begin{proof}
	\textit{(1).} We proceed via induction. For $i=1$, the relation holds due to $\mc{M}_{n}$ being in canonical form. Now suppose that the relation $\bX^i\bY+(-1)^{i+1}\bY\bX^i=\mbf 0$ holds in $\mc{M}_n$ for some $i\in \{1,\dotsc,n-2\}$. Multiplying $\bX\bY+\bY\bX=\mbf 0$ by $\bX^{i}$ from the left we obtain that 
	$$
	\mbf 0=\bX^{i+1}\bY+\bX^{i}\bY\bX = \bX^{i+1}\bY+(-1)^{i+2}\bY\bX^{i+1},
	$$
	where we use the inductive hypothesis for the second equality. 
	By RG, the relation $\bX^{i+1}\bY+(-1)^{i+2}\bY\bX^{i+1}$ also holds in $\mc{M}_{n}$, and hence the statement is proved.

\textit{(2).} Consider a column indexed by a monomial $\bX^{i_0}\bY^{j_1}\bX^{i_1}\bY^{j_2}\cdots \bX^{i_k}\bY^{i_{k+1}}$ where $k\in \NN$, $i_0,j_{k+1}\in\NN\cup \{0\}$ and $i_1, j_1,\ldots, i_k, j_k\in \NN$. Using (1), we know that such a column
is equal to the $\pm 1$ multiple of the column indexed by the monomial 
$\displaystyle\bX^{\sum_{\ell=0}^k i_\ell} \bY^{\sum_{\ell=1}^{k+1} j_\ell}$. 
By using one of the relations \eqref{second-column-relation}, 
the column 
$\displaystyle\bX^{\sum_{\ell=0}^k i_\ell} \bY^{\sum_{\ell=1}^{k+1} j_\ell}$
becomes a linear combination of the columns of the form $\bX^{i}$ and $\bX^{i-1}\bY$ with $i\leq n$.  
\end{proof}

Before proving our main result, the next two lemmas illustrate some properties of the moments in our setting. In particular, we show
that many moments obtained from nc atoms in the measure for $\mc M_n$ are 0. 

\begin{lemma}\label{zero-moment-general}
	Suppose that $\mathcal{M}_n$ satisfies the relation $\bX\bY+\bY\bX=\mbf{0}$.
	If $\beta$ admits a nc measure, then there exists a measure in which every nc atom is of the form 
	\begin{equation}\label{form-of-nc-atoms-general}
		\widetilde{X}=\left(\begin{matrix} \mbf 0_t & B \\ B^t & -\mbf 0_t \end{matrix}\right),
		\quad \widetilde{Y}=\left(\begin{matrix} \mu I_t & \mbf 0_t\\ \mbf 0_t & -\mu I_t \end{matrix}\right),
	\end{equation}
	with $(\widetilde{X},\widetilde{Y})\in (\SSS\RR^{2t\times 2t})^{2}$, $t\in \NN$, $B\in \RR^{t\times t}$, $\mu>0$.
Moreover, every such atoms satisfies:
\begin{enumerate}
	\item\label{moments-general-pt1} $\beta^{(\widetilde{X},\widetilde{Y})}_{X^{2i+1}}=0$ for every $i\in \NN$ such that $2i+1\leq 2n$.
	\item\label{moments-general-pt2} $\beta^{(\widetilde{X},\widetilde{Y})}_{X^jY}=0$ for every $j\in \NN\cup\{0\}$ such that $j+1\leq 2n$.
	\item\label{moments-general-pt4} $\beta^{(\widetilde{X},\widetilde{Y})}_{X^kY^2}=0$ 
		for every odd $k\in \NN$.
\end{enumerate}
\end{lemma}

\begin{proof}  
Since $\mc{M}_{n}$ satisfies $\bX\bY+\bY\bX=\mbf{0}$, by \cite[Proposition 5.1]{BZ18} there exists a measure in which every nc atom is of the form
$(\widetilde{X},\widetilde{Y})\in (\SSS\RR^{2t\times 2t})^{2}$, $t\in \NN$, 
is of the form 	
	\begin{equation*} 
		\widetilde{X}=\left(\begin{matrix} \gamma I_t & B \\ B^t & -\gamma I_t \end{matrix}\right),
		\quad \widetilde{Y}=\left(\begin{matrix} \mu I_t & \mbf 0_t\\ \mbf 0_t & -\mu I_t \end{matrix}\right),
	\end{equation*}
	where $B\in \RR^{t\times t}$, $\gamma\geq 0$, $\mu>0$ (note that \cite[Proposition 5.1]{BZ18} is stated for the case $n=2$, but the proof easily generalizes to $n\in\NN$). 
Moreover, the relation $\bX\bY+\bY\bX=\mbf{0}$ implies that $\gamma=0$ and hence the atoms are of the form \eqref{form-of-nc-atoms-general}.
Let $B_{1} = (BB^{t})$, and $B_{2}=(B^{t}B)$. The following calculations are elementary: 
\begin{gather*}
\widetilde{X}^{2i}=\left(\begin{matrix} B_{1}^i & \mbf 0 \\ \mbf 0 & B_{2}^i \end{matrix}\right),\quad\qquad
\widetilde{X}^{2i}\widetilde{Y}=\left(\begin{matrix} \mu B_{1}^i  & \mbf 0 \\ \mbf 0 & -\mu B_{2}^i	 
\end{matrix}\right),\quad\qquad
 \widetilde{X}^{2i}\widetilde{Y}^2=\left(\begin{matrix} \mu^2 B_{1}^i  & \mbf 0 \\ \mbf 0 & \mu^2 B_{2}^i	 
 \end{matrix}\right),\\
\widetilde{X}^{2i+1}=\left(\begin{matrix} \mbf 0 &  B_{1}^{i} B \\ B_{2}^{i} B^t & \mbf 0
\end{matrix}\right),\quad
\widetilde{X}^{2i+1}\widetilde{Y}=\left(\begin{matrix} \mbf 0 & -\mu B_{1}^{i} B \\ \mu B_{2}^{i-1} B^t & \mbf 0
\end{matrix}\right),\quad
 \widetilde{X}^{2i+1}\widetilde{Y}^2=\left(\begin{matrix} \mbf 0 & \mu^2 B_{1}^{i} B \\ \mu^2 B_{2}^{i} B^t & \mbf 0
\end{matrix}\right).
\end{gather*}

The properties \eqref{moments-general-pt1}-\eqref{moments-general-pt4} are now easy to check, using 
	$$\tr((BB^t)^i)=\tr(B(B^tB)^{i-1}B^t)=\tr((B^tB)^{i-1}B^tB)=\tr((B^tB)^i),$$
where the second equality follows from $\tr(CD)=\tr(DC)$, with $C=B$ and $D=(B^tB)^{i-1}B^t$.
\end{proof}


\begin{lemma}\label{zero-moment-general2}
	Suppose that $\cM_n$ is in the canonical form.
If $\beta$ admits a nc measure, then: 
\begin{enumerate}
	\item\label{moments-cgen-pt1} $\beta_{X^{2i+1}}=\beta_{X}$ for every $i\in \NN$ such that $2i+1\leq 2n$.
	\item\label{moments-cgen-pt2} $\beta_{X^jY}=0$ for every $j\in \NN$ such that $j+1\leq 2n$.
	\item\label{moments-cgen-pt4} $\beta_{X^kY^2}=0$ 
		for every odd $k\in \NN$.
	\item\label{moments-cgen-pt3} When the second relation is:
	\begin{enumerate}
	\item\label{moments-c1-pt3} $\bY^2=\mds 1-\bX^2$, then:  $$\beta_{X^kY^2}=\beta_{X^k}-\beta_{X^{k+2}}\quad \text{for every}\quad k\in \NN\quad \text{such that}\quad k+2\leq 2n.$$
	\item\label{moments-c2/3-pt1} $\bY^2=\mds 1$ or $\bY^2=\mds 1+\bX^2$, then we have that $\beta_{X}=0.$
	\item\label{moments-c2-pt3} $\bY^2=\mds 1$, then:
		$$\beta_{X^kY^2}=\beta_{X^k}\quad \text{for every}\quad k\in \NN\quad \text{such that}\quad k+2\leq 2n.$$
	\item\label{moments--c3-pt3} $\bY^2=\mds 1+\bX^2$, then:  
		$$\beta_{X^kY^2}=\beta_{X^k}+\beta_{X^{k+2}}\quad \text{for every}\quad k\in \NN\quad \text{such that}\quad k+2\leq 2n.$$
	\item\label{moments--c4-pt3} $\bY^2=\bX^2$, then:  
		$$\beta_{X^kY^2}=\beta_{X^{k+2}}\quad \text{for every}\quad k\in \NN\quad \text{such that}\quad k+2\leq 2n.$$
	\end{enumerate}
\end{enumerate}
\end{lemma}

\begin{proof} 
By Theorem \ref{support lemma} \eqref{point-1-support} possible cm atoms in the measure for $\beta$ are:
  \begin{enumerate}
    \item If $\bY^2=\mds 1-\bX^2$: $(1, 0)$, $(-1,0)$, $(0,1)$, $(0,-1)$. 
    \item If $\bY^2=\mds 1$:   $(0,1)$, $(0,-1)$.
    \item If $\bY^2=\mds 1+\bX^2$: $(0,1)$, $(0,-1)$.
    \item If $\bY^2=\bX^2$: $(0, 0)$.
  \end{enumerate}
It is easy to check that the moment matrices $\mc M_n^{(x,y)}$, generated by possible cm atoms  $(x,y)\in \RR^2$, 
satisfy the corresponding relations stated in the lemma. It remains to prove that the nc atoms also satisfy them.
By Lemma \ref{zero-moment-general2} there exist a measure such that in all cases the nc atoms $(\widetilde{X},\widetilde{Y})$ are of the form \eqref{form-of-nc-atoms-general} and satisfy 
\eqref{moments-cgen-pt1}, \eqref{moments-cgen-pt2} and \eqref{moments-cgen-pt4}.
The statement \eqref{moments-c1-pt3} for odd $k\in \NN$ follows by using \eqref{moments-cgen-pt1} and \eqref{moments-cgen-pt4}, while for even $k\in \NN$ it follows
by the following calculation
$$
\beta^{(\widetilde{X},\widetilde{Y})}_{X^{2i}Y^2}= \tr(\widetilde{X}^{2i}\widetilde{Y}^{2})  = \tr(\widetilde{X}^{2i}(I_{2t} - \widetilde{X}^{2})) = \tr(\widetilde{X}^{2i}-\widetilde{X}^{2i+2}) = \tr(\widetilde{X}^{2i}) - \tr(\widetilde{X}^{2i+2}) = \beta^{(\widetilde{X},\widetilde{Y})}_{X^{2i}}-\beta^{(\widetilde{X},\widetilde{Y})}_{X^{2i+2}}
$$
where we used that $\widetilde{Y}^2=I_{2t}-\widetilde{X}^2$ for the second equality.
The statement \eqref{moments-c2/3-pt1} is clear for the nc atoms.
The statement \eqref{moments-c2-pt3} follows by $\widetilde{X}^k\widetilde{Y}^2=\widetilde{X}^{k}$, since $\widetilde{Y}^2=I_{2t}$. 
The statement \eqref{moments--c3-pt3} for odd $k\in \NN$ follows by using \eqref{moments-cgen-pt1}, \eqref{moments-cgen-pt4} and \eqref{moments-c2/3-pt1}, while for even $k\in \NN$ it follows
by the following calculation 
$$
\beta^{(\widetilde{X},\widetilde{Y})}_{X^{2i}Y^2}= \tr(\widetilde{X}^{2i}\widetilde{Y}^{2})  = \tr(\widetilde{X}^{2i}(I_{2t}+ \widetilde{X}^{2})) = \tr(\widetilde{X}^{2i}+\widetilde{X}^{2i+2}) = \tr(\widetilde{X}^{2i})+ \tr(\widetilde{X}^{2i+2}) = \beta^{(\widetilde{X},\widetilde{Y})}_{X^{2i}}+\beta^{(\widetilde{X},\widetilde{Y})}_{X^{2i+2}}
$$
where we used that $\widetilde{Y}^2=I_{2t}+\widetilde{X}^2$ for the second equality.
The statement \eqref{moments--c4-pt3} follows by $\widetilde{X}^k\widetilde{Y}^2=\widetilde{X}^{k+2}$, since $\widetilde{Y}^2=\widetilde{X}^2$. 
This proves the lemma.
\end{proof}

\begin{proof}[Proof of Theorem \ref{M_n-XY+YX=0-all-in-one}]
We can assume WLOG that $\mc{M}_n$ is in the canonical form since the moment matrix admits a measure
if and only if its canonical form admits a measure. 
We rearrange the columns of $\mc{M}_n$ to the order 
	$$\{\mds 1,\bX,\bX^2,\ldots,\bX^n,\bY,\bX\bY,\bX^2\bY,\ldots,\bX^{n-1}\bY\}.$$
The rearranged moment matrix has the form
	\begin{equation}\label{M-c1}
		\widetilde{\cM_n}(\beta_1,\beta_X,\beta_Y):=\left(\begin{matrix}
		\cM_n(\beta_1,\beta_X,X)& B(\beta_Y)\\
		B(\beta_Y) & \cM_n(\beta_1, Y)
	\end{matrix}\right).
	\end{equation}
There are four cases to consider, each corresponding to a relation of \eqref{second-column-relation}. We present in detail the proof when we have relations $\bX\bY+\bY\bX=\mbf{0}$ and $\bY^2=\mbf 1-\bX^2$. The other three cases are argued similarly, and the details can be found in Appendix \ref{proof-rank5-extension}.  \\ 

Given the relations $\bX\bY+\bY\bX=\mbf{0}$ and $\bY^2=\mbf 1-\bX^2$, by Lemma \ref{zero-moment-general2}, the matrices 
$\cM_n(\beta_1,\beta_X,X)$, $\cM_n(\beta_1,Y)$ and $B(\beta_Y)$ are of the forms

\begin{equation*}
\begin{blockarray}{cccccccccccc}
& \mds{1}&\bX&\bX^2&\bX^3&\cdots&\bX^{2k}&\bX^{2k+1}&\cdots&\bX^n\\
\begin{block}{c(ccccccccccc)}
\mds{1}& \beta_1 & \beta_X & \beta_{X^2}& \beta_X & \cdots & \beta_{X^{2k}} & \beta_X & 
	\cdots & c_n \beta_{X^n}+(1-c_{n})\beta_X \\\
\bX& \beta_X & \beta_{X^2} & \beta_X & \beta_{X^4} & \cdots & \beta_X & \beta_{X^{2k+2}}& \cdots & 
	c_{n+1}\beta_{X^{n+1}}+(1-c_{n+1})\beta_X \\\
\bX^2& \beta_{X^2} & \beta_X & \beta_{X^4} &  \beta_X & \cdots  & \beta_{X^{2k+2}} & \beta_X & \cdots&
	c_n \beta_{X^{n+2}} +(1-c_{n})\beta_X \\\
\bX^3&  \beta_{X} & \beta_{X^4} & \beta_X & \beta_{X^6} & \cdots  & \beta_X & \beta_{X^{2k+4}} & \cdots & 
	c_{n+1}\beta_{X^{n+3}} +(1-c_{n+1})\beta_X \\\
\vdots& \vdots &&&&&&&&\vdots \\
\bX^n& c_n\beta_{X^n} +(1-c_{n})  \beta_X &\cdots &\cdots&\cdots&\cdots&\cdots&\cdots&\cdots&
	\beta_{X^{2n}}\\
\end{block}
\end{blockarray},
\end{equation*}
\begin{equation*}
\begin{blockarray}{cccccccccc}
& \bY& \bX\bY&\cdots&
\bX^{2k}\bY&\bX^{2k+1}\bY&\cdots&\bX^{n-1}\bY\\
\begin{block}{c(ccccccccc)}
\bY& \beta_1-\beta_{X^2} & 0  & \cdots & \beta_{X^{2k}}-\beta_{X^{2k+2}} & 0 & 
	\cdots & \cdots\\
\bX\bY& 0 & \beta_{X^2}-\beta_{X^4}  & \cdots & 0 & \beta_{X^{2k+2}}-\beta_{X^{2k+4}}& \cdots & 
	\vdots\\
\vdots& \vdots &&&&&&\vdots \\
\bX^{2k}\bY& \beta_{X^{2k}} - \beta_{X^{2k+2}}&0 &\cdots  & \beta_{X^{4k}}-\beta_{X^{4k+2}} & 0 & \cdots&
	\vdots\\
\bX^{2k+1}\bY & 0 & \beta_{X^{2k+2}}-\beta_{X^{2k+4}} &\cdots & 0 & \beta_{X^{4k+2}}-\beta_{X^{4k+4}} & \cdots & 
	\vdots\\
\vdots& \vdots &&&&&&\vdots \\
\bX^{n-1}\bY& \cdots&\cdots&\cdots&\cdots&\cdots&\cdots&\cdots \\
\end{block}
\end{blockarray}
\end{equation*}
and
\begin{equation}\label{B(beta-y)}
B(\beta_Y):=
\begin{blockarray}{cccccc}
& \bY& \bX\bY&\bX^2\bY&\cdots&\bX^{n-1}\bY\\
\begin{block}{c(ccccc)}
\mds 1& \beta_Y & 0 & 0 & \cdots & 0\\
\bX& 0 & 0 & 0 & \cdots & 0\\
\vdots& \vdots &&&&\vdots\\
\bX^n& 0 & 0 & 0 & \cdots & 0\\
\end{block}
\end{blockarray}, 
\end{equation}
respectively, 
where $c_m=\frac{(-1)^{m}+1}{2}$.
By Lemma \ref{zero-moment-general} the nc atoms must be of the form
\eqref{form-of-nc-atoms-general}.
Hence the only  way 
to cancel the odd moment, $\beta_{X}$ in $\cM_n(\beta_1,\beta_X,X)$ 
and $\beta_Y$ moment in 
$B(\beta_Y)$, is by using atoms of size 1, which are $(\pm 1, 0)$ and $(0,\pm 1)$.\\

\noindent \textbf{Claim.}
We have that
	$$|\beta_X|\widetilde{\cM}_{n}^{(\sign(\beta_X)1,0)}+|\beta_Y|\widetilde{\cM}_{n}^{(0,\sign(\beta_Y)1)} 
	\preceq 
	\gamma_1\widetilde{\cM}_{n}^{(1,0)}+
	\gamma_2\widetilde{\cM}_{n}^{(-1,0)}+
	\delta_1\widetilde{\cM}_{n}^{(0,1)}+
	\delta_2\widetilde{\cM}_{n}^{(0,-1)}$$
for every $\gamma_1,\gamma_2,\delta_1,\delta_2\geq 0$ such that $\gamma_1-\gamma_2=\beta_X$
and $\delta_1-\delta_2=\beta_Y$.\\

We consider four cases depending on the signs of $\beta_X$ and $\beta_Y$. 
If $\beta_X\geq 0$, then $\sign(\beta_X)1=1$ and hence $\gamma_1\geq \beta_X.$
Else $\beta_X< 0$, $\sign(\beta_X)1=-1$ and hence $\gamma_2\geq |\beta_X|.$
Thus, 
  \begin{equation} \label{ineq1-c1}
    |\beta_X|\widetilde{\cM}_{n}^{(\sign(\beta_X)1,0)} \preceq \gamma_1 \widetilde{\cM}_{n}^{(1,0)}+
     \gamma_2 \widetilde{\cM}_{n}^{(-1,0)}.
  \end{equation}
Similarly,
  \begin{equation} \label{ineq2-c1}
	|\beta_Y|\widetilde{\cM}_{n}^{(0,\sign(\beta_Y)1)} \preceq \delta_1 \widetilde{\cM}_{n}^{(0,1)}+
 	\delta_2 \widetilde{\cM}_{n}^{(0,-1)}.
  \end{equation}
Now, \eqref{ineq1-c1} and \eqref{ineq2-c1} imply the claim.\\

By the claim it follows that $\Big(|\beta_X|\widetilde{\cM}_{n}^{(\sign(\beta_X)1,0)}+|\beta_Y|\widetilde{\cM}_{n}^{(0,\sign(\beta_Y)1)} \Big)$ is the smallest matrix (under the L{\"o}wner partial ordering) such that, $\widetilde{\cM_n}(\beta_1,\beta_X,\beta_Y)$ admits a measure if and only if 
	$$\widetilde{\cM_n}(\beta_1-\beta_X-\beta_Y,0,0)=
	\widetilde{\cM_n}(\beta_1,\beta_X,\beta_Y)-
	\Big(|\beta_X|\widetilde{\cM}_{n}^{(\sign(\beta_X)1,0)}+
	|\beta_Y|\widetilde{\cM}_{n}^{(0,\sign(\beta_Y)1)}\Big)$$
	admits a measure.
Now observe that the existence of a measure for $\cM_n(\beta_1-\beta_X-\beta_Y,0,X)$, i.e.,
\begin{equation*}
\begin{blockarray}{ccccccccc}
& \mds{1}&\bX&\bX^2&\bX^3&\cdots&\bX^n\\
\begin{block}{c(cccccccc)}
\mds{1}& \beta_1-|\beta_X|-|\beta_Y| & 0 & \beta_{X^2}-|\beta_X|& 0 & 
	\cdots & c_n (\beta_{X^n}-|\beta_X|)  \\\
\bX& 0 & \beta_{X^2}-|\beta_X| & 0 & \beta_{X^4}-|\beta_X| & \cdots & 
	c_{n+1}(\beta_{X^{n+1}}-|\beta_X|)  \\\
\bX^2& \beta_{X^2}-|\beta_X| & 0 & \beta_{X^4}-|\beta_X| &  0 &  \cdots&
	c_n (\beta_{X^{n+2}}-|\beta_X|)  \\\
\bX^3& 0 & \beta_{X^4}-|\beta_X| & 0 & \beta_{X^6}-|\beta_X| & \cdots & 
	c_{n+1}(\beta_{X^{n+3}}-|\beta_X|)  \\\
\vdots& \vdots &&&&&\vdots \\
\bX^n& c_n(\beta_{X^n}-|\beta_X|)  &\cdots &\cdots&\cdots&\cdots&
	\beta_{X^{2n}}-|\beta_X|\\
\end{block}
\end{blockarray},
\end{equation*}
with support a subset of $[-1,1]$ is the truncated Hausdorf moment problem. 
Hence by \cite[Theorem III.2.3]{KN77}, the matrix $\cM_n(\beta_1-\beta_X-\beta_Y,0,X)$
admits a measure if and only if it is psd and 
\begin{equation*}
\begin{blockarray}{cccccccccc}
& \bY& \bX\bY&\cdots&
\bX^{2k}\bY&\bX^{2k+1}\bY&\cdots&\bX^{n-1}\bY\\
\begin{block}{c(ccccccccc)}
\bY& \beta_1-\beta_{X^2}-|\beta_Y| & 0  & \cdots & \beta_{X^{2k}}-\beta_{X^{2k+2}} & 0 & 
	\cdots & \cdots\\
\bX\bY& 0 & \beta_{X^2}-\beta_{X^4}  & \cdots & 0 & \beta_{X^{2k+2}}-\beta_{X^{2k+4}}& \cdots & 
	\vdots\\
\vdots& \vdots &&&&&&\vdots \\
\bX^{2k}\bY& \beta_{X^{2k}} - \beta_{X^{2k+2}}&0 &\cdots  & \beta_{X^{4k}}-\beta_{X^{4k+2}} & 0 & \cdots&
	\vdots\\
\bX^{2k+1}\bY & 0 & \beta_{X^{2k+2}}-\beta_{X^{2k+4}} &\cdots & 0 & \beta_{X^{4k+2}}-\beta_{X^{4k+4}} & \cdots & 
	\vdots\\
\vdots& \vdots &&&&&&\vdots \\
\bX^{n-1}\bY& \cdots&\cdots&\cdots&\cdots&\cdots&\cdots&\cdots \\
\end{block}
\end{blockarray},
\end{equation*}
which is exactly 
$\cM_n(\beta_1-\beta_Y,Y)$, 
is psd.
Now note that if $x_i$, $i=1,\ldots, k$, $k\in \NN$, are atoms in the measure for 
$\cM_n(\beta_1-\beta_X-\beta_Y,0,X)$ with the corresponding densities $\mu_i$,  $i=1,\ldots, k$, then
	$$
	\Big(\left(\begin{matrix}
	0&x_i\\
	x_i&0
	\end{matrix}\right),
	\left(\begin{matrix}
	\sqrt{1-x_i^2}&0\\
	0&-\sqrt{1-x_i^2}
	\end{matrix}\right)
	\Big),     \quad  i=1,\ldots,k,
	$$
with densities $\mu_i$, $i=1,\ldots,k$, are atoms which represent $\widetilde{\cM_n}(\beta_1-\beta_X-\beta_Y,0,0)$. 
\end{proof}

\section{Reducing the degenerate truncated hyperbolic moment problem}
\label{S4}

Prompted by the outcomes of the previous section (proof of Theorem \ref{M_n-XY+YX=0-all-in-one}), we use the reduction technique to present a simplified proof one of the main results in \cite{CF05}, the degenerate truncated hyperbolic moment problem, i.e., when $\mc{M}_{n}$ is commutative and satisfies $\bX\bY = \mbf{0}$. 

\begin{remark}
	 \ab{Curto and Fialkow have previously used the reduction technique for the complex moment problem when $Z = \bar{Z}$, and shown how the truncated complex moment problem with this column relation is equivalent to the truncated Hamburger moment problem (see the discussion after \cite[Conjecture 3.16]{CF96}).}
\end{remark}


\begin{theorem}\cite[Theorem 3.1]{CF05} \label{XY=0-cm}
	Let $\cM_n$ be a moment matrix satisfying the relations $\bX\bY=\bY\bX=\mbf 0$. 
	If $\cM_n$ is positive semidefinite, recursively generated and satisfies 
	$\Rank(\cM_n) \leq \Card (\mathcal V)$, where 
		 	$$\mathcal V:=\bigcap_{
			\substack{g\in \RR[X,Y]_{\leq 2},\\ 
						g(\bX,\bY)=\mbf 0\;\text{in}\;\cM_n}}
				\left\{ (x,y)\in \RR^2\colon g(x,y)=0 \right\},$$ 
	then it admits a representing measure.
	Moreover, if $\Rank(\cM_n)\leq 2n$, then $\mc M_n$ admits a $(\Rank(\cM_n))$-atomic measure, and
	if $\Rank(\cM_n)=2n+1$, then $\mc M_n$ admits a $(2n+1)$- or $(2n+2)$-atomic measure.
\end{theorem}

\begin{proof} 
Note that the basis for $\mathcal C_{\mc M_n}$ is a subset of $\{\mds 1, \bX,\ldots,\bX^n,\bY,\ldots,\bY^n\}$.
Reordering the columns to
$$\mds 1,\bX,\bX^2,\ldots,\bX^n,\bY,\ldots,\bY^n,\bX\bY,\ldots, \bX\bY^{n-1},\bX^2\bY,\ldots,\bX^2\bY^{n-2},\ldots,\bX^{n-1}\bY,$$
we have that $\mc M_n=M\oplus \mbf{0}$ where 
	$$M=
\begin{mpmatrix} 
1 & a^t & b^t  \\
a & A & \mbf{0} \\
b & \mbf{0} & B \\
 \end{mpmatrix}, \quad
A=\begin{mpmatrix} 
\beta_{X^2} & \ldots & \beta_{X^{n+1}}\\
\vdots & \ddots & \vdots\\
\beta_{X^{n+1}}& \cdots & \beta_{X^{2n}}
\end{mpmatrix},\quad
B=\begin{mpmatrix} 
\beta_{Y^2} & \ldots & \beta_{Y^{n+1}}\\
\vdots & \ddots & \vdots\\
\beta_{Y^{n+1}}& \cdots & \beta_{Y^{2n}}
\end{mpmatrix},\quad
a=\begin{mpmatrix} \beta_X \\ \vdots \\ \beta_{X^n}\end{mpmatrix},\quad
b=\begin{mpmatrix} \beta_Y \\ \vdots \\ \beta_{Y^n}\end{mpmatrix}.
$$
We separate two cases according to the rank of $\mc M_n$.\\

\noindent\textbf{Case 1:} $\Rank(\cM_n)=2n+1.$ 
From $M\succ 0$ it follows that the Schur complement  $\eta:=1-a^t A^{-1} a -b^t B^{-1}b>0$ of the block $A\oplus B$ is positive.
For $\alpha:=a^tA^{-1}a+\frac{\eta}{2}$ we have that
$1-\alpha=b^tB^{-1}b+\frac{\eta}{2}$ and 
$\begin{mpmatrix} \alpha & a^t \\a & A \end{mpmatrix}$, 
$\begin{mpmatrix} 1-\alpha & b^t \\b & B\end{mpmatrix}$ are both positive definite.
By \cite[Theorem 3.9]{CF91} they admit a measure consisting of $n+1$ atoms $x_0,\ldots,x_n$ and $y_0,\ldots,y_n$, respectively.
So $\mc M_n$ admits a measure consisting of at most $2n+2$ atoms $(x_0,0),\ldots,(x_n,0), (0,y_0),\ldots,(0,y_n)$, with only one potential duplication, namely $(x_i,0)=(0,y_j)=(0,0)$ for some $i,j$. \\

\noindent\textbf{Case 2:} $\Rank(\cM_n)\leq 2n.$
Let $k_1:=\Rank A=\Rank A_{k_1}$ and $k_2:=\Rank B=\Rank B_{k_2}$, where $A_{k_1}, B_{k_2}$ are the leading principal
submatrices of size $k_1$, $k_2$ of $A, B$, and the second equalities follow from $\mc M_n$ being RG.
We denote by $a_{k_1}$, $b_{k_2}$ the restrictions of $a, b$ to the first $k_1$, $k_2$ rows, respectively. We write 
$M_{k_1,k_2}:=\begin{mpmatrix} 1 & a_{k_1}^t & b_{k_2}^t \\ a_{k_1} & A_{k_1} & 0\\ b_{k_2} & 0 & B_{k_2} \end{mpmatrix}$.
We separate two cases according to the difference
$\left(\Rank (\mc M_n)-\Rank\begin{mpmatrix} A & \mbf 0 \\ \mbf 0 & B\end{mpmatrix}\right)$.

\noindent\textbf{Case 2.1:} $\left(\Rank (\mc M_n)-\Rank\begin{mpmatrix} A & \mbf 0 \\ \mbf 0 & B\end{mpmatrix}\right)=1$.
We have $k_1<n$ or $k_2<n$. We may assume WLOG that $k_1<n$.
In $A$ we have 
$\bX^{k_1+1}=\sum_{i=1}^{k_1}\gamma_i \bX^i$ for some $\gamma_i\in \RR$.
By \cite[Proposition 3.9]{CF96}, $\bX^{k_1+1}=\sum_{i=1}^{k_1}\gamma_i \bX^i$ holds also in $\mc M_n$.
Hence $\gamma_1\neq 0$, since otherwise 
$A_{k_1}=[\mc M_n]_{\{\bX,\ldots,\bX^{k_1}\}}$
is singular, which contradicts $\Rank A_{k_1}=k_1$.
In
$\begin{mpmatrix} \ast & a^t \\ a & A \end{mpmatrix}$
we have 
$$
[\bX^{k_1}]_{\{\{\mds{1},\ldots,\bX^{n}\},\{\bX,\ldots,\bX^n\}\}}=\sum_{i=0}^{k_1-1}\gamma_{i+1} [\bX^i]_{\{\{\mds{1},\ldots,\bX^{n}\},\{\bX,\ldots,\bX^n\}\}}.
$$
Since $\gamma_1\neq 0$ there is a unique value of $\ast$ such that 
$\bX^{k_1}=\sum_{i=0}^{k_1-1}\gamma_{i\ab{+1}} \bX^i$ and $\Rank\begin{mpmatrix} \ast & a_{\ab{k_{1}}}^t \\ a_{\ab{k_{1}}} & A_{\ab{k_{1}}} \end{mpmatrix}=\ab{k_1}$, this is given by $\ast:=a_{k_1}^t A_{k_1}^{-1}a_{k_1}$, making the matrix 
$$
\begin{pmatrix} a_{k_1}^t A_{k_1}^{-1}a_{k_1} & a^t \\ a & A \end{pmatrix},
$$
psd and RG.
Since the Schur complement of $M_{k_1,k_2}$
is positive $\left(1-a_{k_1}^t A_{k_1}^{-1}a_{k_1}-b_{k_2}^t B_{k_2}^{-1}b_{k_2}>0\right)$, we have that $1-a_{k_1}^t A_{k_1}^{-1}a_{k_1}>b_{k_1}^t B_{k_1}^{-1}b_{k_2}$ and hence
(again by Schur complements) the matrix
$$
\begin{pmatrix} 1-a_{k_1}^t A_{k_1}^{-1}a_{k_1} & b_{k_2}^t \\ b_{k_1} & B_{k_1} \end{pmatrix},
$$
is positive definite.
By \cite[Theorem 3.9]{CF91} both, 
$$
\begin{pmatrix} a_{k_1}^t A_{k_1}^{-1}a_{k_1} & a^t \\ a & A \end{pmatrix}, \quad \text{and } \
\begin{pmatrix} 1-a_{k_1}^t A_{k_1}^{-1}a_{k_1} & b^t \\ b & B \end{pmatrix},
$$
admit a $k_1$- and $(k_2+1)$-atomic measures, respectively. Hence $\mc M_n$ admits a $\Rank \mc M_n$-atomic measure.\\

\noindent\textbf{Case 2.2:} $\left(\Rank (\mc M_n)-\Rank\begin{mpmatrix} A & \mbf 0 \\ \mbf 0 & B\end{mpmatrix}\right)=0$.
The Schur complement $1-a_{k_1}^t A_{k_1}^{-1}a_{k_1}-b_{k_2}^t B_{k_2}^{-1}b_{k_2}$ of the block $A_{k_1}\oplus B_{k_2}$ in $M_{k_1,k_2}$
is equal to zero, thus 
$$
M_{k_1,k_2}=
\begin{pmatrix} a_{k_1}^t A_{k_1}^{-1}a_{k_1}  & a_{k_1}^t & 0\\ a_{k_1} & A_{k_1} & 0 \\ 0&0&0 \end{pmatrix}+
\begin{pmatrix} b_{k_2}^t B_{k_2}^{-1}b_{k_2} & 0 & b_{k_2}^t \\ 0 & 0 & 0 \\ b_{k_2} & 0 & B_{k_2} \end{pmatrix}.
$$
If $k_1<n$, then as in Case 2.1 we see that
$\begin{mpmatrix} a_{k_1}^t A_{k_1}^{-1}a_{k_1}  & a^t \\ a & A \end{mpmatrix}$
is psd, RG and of rank $k_1$ (similarly if $k_{2}<n$). 
Let us now assume that $k_1=n$.
Then the matrix 
$$
U:=\begin{pmatrix} a^t A^{-1}a  & a^t \\ a & A \end{pmatrix},
$$
is  psd and of rank $n$. Let $U_j$ be the $j$-th column of $U$. 
Suppose there is a nontrivial linear combination $\displaystyle \mbf{0}=U_1+\textstyle\sum\nolimits_{i=2}^{i_{0}} \delta_i U_i$ where 
$\delta_i\in \RR$, $i_0\leq n$ and $\delta_{i_0}\neq 0$. Observe also the matrix
$$
V:=\begin{pmatrix} b_{k_2}^t B_{k_2}^{-1}b_{k_2} & b_{k_2}^t \\ b_{k_2} & B_{k_2} \end{pmatrix},
$$
is psd, of rank $k_2$, and there is a nontrivial linear combination
$\textstyle\mbf 0=V_1+\sum\nolimits_{j=2}^{k_{2}+1} \zeta_j V_j$ where
$\zeta_j\in \RR$ and $V_j$ is the $j$-th column of $V$. 
Therefore 
\begin{equation}\label{col-rel}
\mbf 0=\begin{pmatrix}a ^t A^{-1}a+b_{k_2}^t B_{k_2}^{-1}b_{k_2} \\ a\\ b_{k_2}\end{pmatrix}+
		\sum_{\ab{2}\leq i\leq i_0} \delta_i \begin{pmatrix}U_i\\ 0\end{pmatrix}
		+\sum_{\ab{2}\leq j\leq k_2+1} \zeta_j \begin{pmatrix}v_{1j}\\ 0\\ V_j'\end{pmatrix},
\end{equation}
where $V_j=\begin{pmatrix}v_{1j}& V_j'\end{pmatrix}^{t},$ $v_{1j}\in \RR$, $V_j'\in \RR^{k_2}$. By \cite[Proposition 3.9]{CF96}, \eqref{col-rel} implies that $\mc M_n$ must satisfy the column relation
	$$
	\displaystyle\mbf 0=\mds 1+
		\sum_{1\leq i\leq i_0-1} \delta_i \bX^i
		+\sum_{1\leq j\leq k_2} \zeta_j \bY^j.
	$$
But then $\Card(\mathcal V)\leq i_0-1+k_2$, which implies
$$
\Card (\mathcal V)\leq n-1+k_2 < n+k_2=\Rank (\cM_n),
$$
a contradiction with the assumption $\Rank(\mc{M}_{n})\leq \Card(\mc{V})$.
Hence $U_{n+1}\in \Span\{U_1,\ldots,U_{n}\}$ and $U$ is RG. Similarly, $V$ is RG for every $k_2$. 
By \cite[Theorem 3.9]{CF91}, both 
$$
\begin{pmatrix} a_{k_1}^t A_{k_1}^{-1}a_{k_1}  & a^t \\ a & A \end{pmatrix}, \quad \text{and } \
\begin{pmatrix} b_{k_2}^t B_{k_2}^{-1}b_{k_2}  & b^t \\ b & B \end{pmatrix}
$$ 
admit a $k_1$- and $k_2$-atomic measure, respectively, and $\mc M_n$ admits a $(\Rank (\mc M_n))$-atomic measure.
\end{proof}

\begin{remark}
The matrix $\cM_n$ of $\Rank \mc M_n=2n+1$ satisfying the assumptions of Theorem \ref{XY=0-cm} admits a $(2n+1)$-atomic if and only if
one of $Z_1:=\begin{mpmatrix} a^tA^{-1}a & a^t \\a & A \end{mpmatrix}$ or $Z_2:=\begin{mpmatrix} b^tB^{-1}b & b^t \\b & B \end{mpmatrix}$ is RG, i.e., the last column is in the span of the others. Indeed, if $Z_1$ is RG then it admits a $n$-atomic measure and 
$\begin{mpmatrix} 1-b^tB^{-1}b & a^t \\a & A \end{mpmatrix}$ being positive definite admits a $(n+1)$-atomic measure which gives a 
$(2n+1)$-atomic measure for $\mc M_n$. Similarly for the pair $Z_2$ and $\begin{mpmatrix} 1-a^tA^{-1}a & b^t \\b & B \end{mpmatrix}$. If $Z_1$ and $Z_2$ are not RG, and $\mc{M}_{n}$ admits a $(2n+1)$-atomic measure, there must 
exist an $\alpha\in (0,1)$ such that $\alpha>a^tA^{-1}a$, $1-\alpha>b^t B^{-1}b$, and both matrices $\begin{mpmatrix} \alpha & a^t \\a & A \end{mpmatrix}$ and $\begin{mpmatrix} 1-\alpha & b^t \\b & B\end{mpmatrix}$
admit a $(n+1)$-atomic measures with the shared atom $(0,0)$. But then removing $(0,0)$ as an atom of both we are 
left with rank $n$ matrices $Z_1$ and $Z_2$, both admitting a measure. Hence they should be RG which would be a contradiction.
\end{remark}


\appendix
\section{Direct calculations for some results from the manuscript}

\subsection{Transformations for Lemma \ref{linear-trans}}\label{append-transforms}

Firstly, note that all the square roots are well-defined which follows from the fact that $\mc M_2$ is psd (for details see the proof of \cite[Proposition 4.1 (1)]{BZ18}).
We separate 5 cases according to $d\in \RR$.\\

\noindent{\textbf{Case 2.1:} $d<-2$.}

$$
\begin{array}{|c|c|c|} 
\hline
\text{Transformation }(x,y)\mapsto & \text{The first relation of }\mc M_2 & \text{The second relation of }\mc M_2\\
\hline
(x+y,y-x) & (2-d)\bX^2-(2+d)\bY^2=(4a-2d)\mds 1 & \bX\bY+\bY\bX=2\mds 1\\
\hline
\left(\sqrt{2-d}x,\sqrt{-2-d}y\right) & \bX^2+\bY^2=(4a-2d)\mds 1 & \bX\bY+\bY\bX=2\sqrt{d^2-4}\mds 1\\
\hline
\left(\frac{1}{\sqrt{4a-2d}}x,\frac{1}{\sqrt{4a-2d}}y\right) & \bX^2+\bY^2=\mds 1 & 
\bX\bY+\bY\bX= \frac{\sqrt{d^2-4}}{2a-d} \mds 1=:\widehat a \mds 1\\
\hline
\left( x,x+y\right) & 
\bX\bY+\bY\bX=\widehat a \mds 1+2\bX^2 & \bY^2=(1+\widehat a)\mds 1\\
\hline
\left( x,\frac{1}{\sqrt{1+\widehat a}}y\right) & 
\bX\bY+\bY\bX=\frac{\widehat{a}}{\sqrt{1+\widehat a}}\mds 1+\frac{2}{\sqrt{1+\widehat a}}\bX^2 & \bY^2=\mds 1\\
\hline
\left( -x+\frac{1}{\sqrt{1+\widehat a}}y,y\right) & 
\bX^2=\Big( \frac{1+\widehat a}{4}-\frac{2\widehat a}{ 1+\widehat a}\Big)=:\widetilde{a}\mds 1 & \bY^2=\mds 1\\
\hline
\left(\frac{x}{\sqrt{\widetilde a}},y\right) & 
\bX^2= \mds 1 & \bY^2=\mds 1\\
\hline
\left(\frac{x+y}{2},\frac{y-x}{2}\right) & 
\bX\bY+\bY\bX=\mbf 0 & \bX^2+\bY^2=\mds 1\\
\hline
\end{array}
$$\\

\noindent{\textbf{Case 2.2:} $d=-2$.}

$$
\begin{array}{|c|c|c|} 
\hline
\text{Transformation }(x,y)\mapsto & \text{The first relation of }\mc M_2 & \text{The second relation of }\mc M_2\\
\hline
(x+y,y-x) & \bX^2=(a+1)\mds 1 & \bX\bY+\bY\bX=2\mds 1\\
\hline
\left(\frac{1}{\sqrt{a+1}}x,y\right) & \bX^2=\mds 1 & \bX\bY+\bY\bX=\frac{2}{\sqrt{a+1}}\mds 1\\
\hline
\left(y,x\right) & \bY^2=\mds 1 &  \bX\bY+\bY\bX=\frac{2}{\sqrt{a+1}}\mds 1\\
\hline
\left( x-\frac{1}{\sqrt{a+1}}y,y\right) &  \bY^2=\mds 1 &  \bX\bY+\bY\bX=\mbf 0\\ 
\hline
\end{array}
$$\\

\noindent{\textbf{Case 2.3:} $-2<d<2$.}

$$
\begin{array}{|c|c|c|} 
\hline
\text{Transformation }(x,y)\mapsto & \text{The first relation of }\mc M_2 & \text{The second relation of }\mc M_2\\
\hline
(x+y,y-x) & (2-d)\bX^2-(2+d)\bY^2=(4a-2d)\mds 1 & \bX\bY+\bY\bX=2\mds 1\\
\hline
\left(\sqrt{2-d}x,\sqrt{2+d}y\right) & \bX^2-\bY^2=(4a-2d)\mds 1 & \bX\bY+\bY\bX=2\sqrt{4-d^2}\mds 1\\
\hline
\multicolumn{3}{|c|}{\text{We may assume that } 4a-2d\leq 0. \text{ Otherwise we do the transformation }(x,y)\mapsto (y,x)}\\
\hline
\left(x,x-\frac{\sqrt{4-d^2}}{\sqrt{d-2a}}y\right) & \bX^2-\Big(\underbrace{\frac{4(4-d^2)^2}{(2d-4a)^2}-1}_{C}
\Big)\bY^2=\mbf 0
& \bX\bY+\bY\bX=-(2d-4a)\bX^2-(2d-4a)\mds 1\\
\hline
\left( \sqrt{C}x,y\right) & \bY^2-\bX^2=\mbf 0&
\bX\bY+\bY\bX
= -\underbrace{\sqrt{C}(2d-4a)}_{D}\mds 1+\sqrt{C}(2d-4a)\bX^2\\
\hline
\left( x+y,y-x\right) & (2-D)\bX^2-(2+D)\bY^2=-4D\mds 1
&
\bX\bY+\bY\bX=\mbf 0 \\
\hline
\end{array}
$$\\

\noindent{\textbf{Case 2.3.1:} $D=2$.}

$$
\begin{array}{|c|c|c|} 
\hline
\text{Transformation }(x,y)\mapsto & \text{The first relation of }\mc M_2 & \text{The second relation of }\mc M_2\\
\hline
(x,y) & \bX^2=2\mds 1 & \bX\bY+\bY\bX=\mbf 0\\
\hline
\left(\frac{y}{\sqrt{2}},x\right) & \bY^2=\mds 1 & \bX\bY+\bY\bX=\mbf 0\\
\hline
\end{array}
$$\\

\noindent{\textbf{Case 2.3.2:} $D=-2$.}

$$
\begin{array}{|c|c|c|} 
\hline
\text{Transformation }(x,y)\mapsto & \text{The first relation of }\mc M_2 & \text{The second relation of }\mc M_2\\
\hline
(x,y)  & \bY^2=2\mds 1 & \bX\bY+\bY\bX=\mbf 0\\
\hline
\left(x,\frac{y}{\sqrt{2}}\right) & \bY^2=\mds 1 & \bX\bY+\bY\bX=\mbf 0\\
\hline
\end{array}
$$\\

\noindent{\textbf{Case 2.3.3:} $|D|\neq 2$.}

$$
\begin{array}{|c|c|c|} 
\hline
\text{Transformation }(x,y)\mapsto & \text{The first relation of }\mc M_2 & \text{The second relation of }\mc M_2\\
\hline
(\sqrt{|2-D|}x,\sqrt{|2+D|}y) & \pm \bX^2\pm \bY^2=-4D\sqrt{|4-d^2|}\mds 1 & \bX\bY+\bY\bX=\mbf 0\\
\hline
\end{array}
$$\\

\noindent{\textbf{Case 2.3.3.1:}  $\bX^2+\bY^2=\widetilde{A}\mds 1$, $\widetilde A>0$.}\\
$$
\begin{array}{|c|c|c|} 
\hline
\text{Transformation }(x,y)\mapsto & \text{The first relation of }\mc M_2 & \text{The second relation of }\mc M_2\\
\hline
\left(\frac{1}{\sqrt{\widetilde A}}x,\frac{1}{\sqrt{\widetilde A}}y\right) & \bX^2+\bY^2=\mds 1 & \bX\bY+\bY\bX=\mbf 0\\
\hline
\end{array}
$$\\

\noindent{\textbf{Case 2.3.3.2:}  $\bY^2-\bX^2=\widetilde{A}\mds 1$.}\\

We may assume that $\widetilde A\geq 0$ for if not, we may transform 
$(x,y)\mapsto (y,x)$.

$$
\begin{array}{|c|c|c|} 
\hline
\text{Transformation }(x,y)\mapsto & \text{The first relation of }\mc M_2 & \text{The second relation of }\mc M_2\\
\hline
\left(\frac{1}{\sqrt{\widetilde A}}x,\frac{1}{\sqrt{\widetilde A}}y\right) & \bY^2-\bX^2=\mds 1 & \bX\bY+\bY\bX=\mbf 0\\
\hline
\end{array}
$$\\

\noindent{\textbf{Case 2.4:} $2=d$.}

$$
\begin{array}{|c|c|c|} 
\hline
\text{Transformation }(x,y)\mapsto & \text{The first relation of }\mc M_2 & \text{The second relation of }\mc M_2\\
\hline
(x+y,y-x) & \bY^2=(1-a)\mds 1 & \bX\bY+\bY\bX=2\mds 1\\
\hline
\left(x,\frac{1}{\sqrt{1-a}}y\right) & \bY^2=\mds 1 & \bX\bY+\bY\bX=\frac{2}{\sqrt{1-a}}\mds 1\\
\hline 
\left( x-\frac{1}{\sqrt{1-a}}y,y\right) &  \bY^2=\mds 1 &  \bX\bY+\bY\bX=\mbf 0\\ 
\hline
\end{array}
$$\\

\noindent{\textbf{Case 2.5:} $2<d$.}

$$
\begin{array}{|c|c|c|} 
\hline
\text{Transformation }(x,y)\mapsto & \text{The first relation of }\mc M_2 & \text{The second relation of }\mc M_2\\
\hline
(x+y,y-x) & (2-d)\bX^2-(2+d)\bY^2=(4a-2d)\mds 1 & \bX\bY+\bY\bX=2\mds 1\\
\hline
\left(\sqrt{d-2}x,\sqrt{d+2}y\right) & \bX^2+\bY^2=(2d-4a)\mds 1 & \bX\bY+\bY\bX=2\sqrt{d^2-4}\mds 1\\
\hline
\left(\frac{1}{\sqrt{2d-4a}}x,\frac{1}{\sqrt{2d-4a}}y\right) & \bX^2+\bY^2=\mds 1 & 
\bX\bY+\bY\bX= \frac{\sqrt{d^2-4}}{2d-a} \mds 1=:\widehat a \mds 1\\
\hline
\left( x,x+y\right) & 
\bX\bY+\bY\bX=\widehat a \mds 1+2\bX^2 & \bY^2=(1+\widehat a)\mds 1\\
\hline
\left( x,\frac{1}{\sqrt{1+\widehat a}}y\right) & 
\bX\bY+\bY\bX=\frac{\widehat{a}}{\sqrt{1+\widehat a}}\mds 1+\frac{2}{\sqrt{1+\widehat a}}\bX^2 & \bY^2=\mds 1\\
\hline
\left( -x+\frac{1}{\sqrt{1+\widehat a}}y,y\right) & 
\bX^2=\Big( \frac{1+\widehat a}{4}-\frac{2\widehat a}{ 1+\widehat a}\Big)=:\widetilde{a}\mds 1 & \bY^2=\mds 1\\
\hline
\left(\frac{x}{\sqrt{\widetilde a}},y\right) & 
\bX^2= \mds 1 & \bY^2=\mds 1\\
\hline
\left(\frac{x+y}{2},\frac{y-x}{2}\right) & 
\bX\bY+\bY\bX=\mbf 0 & \bX^2+\bY^2=\mds 1\\
\hline
\end{array}
$$

\subsection{Calculations for Lemma \ref{linear-trans-on-zero-moments}\label{calc-for-lemma-zero-moments}}

The statement of the lemma follows by the following calculations:
	\begin{align*}
		\widetilde{\beta}_{X} &= L_{\beta^{(4)}}(bX+cY)=b\beta_X+c\beta_Y=0,\\
		\widetilde{\beta}_{Y} &=  L_{\beta^{(4)}}(eX+fY)=e\beta_X+f\beta_Y=0,\\
		\widetilde{\beta}_{X^3} &= L_{\beta^{(4)}}((bX+cY)^3)=
			 L_{\beta^{(4)}}\left(b^3X^3+b^2c(X^2Y+XYX+YX^2)+bc^2(XY^2+YXY+Y^2X)+c^3Y^3\right)\\
			&=b^3\beta_{X^3}+3b^2c \beta_{X^2Y}+3bc^2\beta_{XY^2}+c^3\beta_{Y^3}=0,\\
		\widetilde{\beta}_{X^2Y} &= L_{\beta^{(4)}}((bX+cY)^2(eX+fY))\\
			&=	 L_{\beta^{(4)}}\left(b^2eX^3+b^2fX^2Y+bce (XYX+ YX^2)+
					bcf(XY^2+YXY)+c^2e Y^2X +c^2fY^3\right)\\
			&=b^2e\beta_{X^3}+(b^2f+2bce) \beta_{X^2Y}+(2bcf+c^2e)\beta_{XY^2}+c^2f\beta_{Y^3}=0,\\
		\widetilde{\beta}_{Y^3} &= L_{\beta^{(4)}}((eX+fY)^3)=
			 L_{\beta^{(4)}}\left(e^3X^3+e^2f(X^2Y+XYX+YX^2)+ef^2(XY^2+YXY+Y^2X)+f^3Y^3\right)\\
			&=e^3\beta_{X^3}+3e^2f \beta_{X^2Y}+3ef^2\beta_{XY^2}+f^3\beta_{Y^3}=0.
	\end{align*}

\subsection{Calculations for Lemma \ref{gen-lemma-one--1}\label{calc-part3-app}}

 	Part \eqref{Y2=1-pt3} of Lemma \ref{gen-lemma-one--1} follows by
 	the following calculations:
	\begin{equation*}
  	\begin{split}
		\widehat{X}^2 &= \widehat{X}^2\widetilde{Y}^2 = \left(\begin{array}{cc} \hat D^2+xx^t & Dx \\ x^t\hat D  & x^tx\end{array}\right),\\	
		\widehat{X}\widetilde{Y}   &= \widehat{X}\widetilde{Y}^3=\left(\begin{array}{cc} \hat D & -x \\ x^t & 0\end{array}\right),\\
		\widehat{X}^3 &= \left(\begin{array}{cc} \hat D^3+xx^t\hat D+\hat Dxx^t & \ast \\ \ast & x^t\hat Dx\end{array}\right),\\
		\widehat{X}^2\widetilde{Y} &= \left(\begin{array}{cc} \hat D^2+xx^t & -\hat Dx \\ x^t\hat D  & -x^tx\end{array}\right),
	\end{split}
	\qquad
	\begin{split}
		\widehat{X}^4 &= \left(\begin{array}{cc} (\hat D^2+xx^t)^2+\hat Dxx^t\hat D & \ast\\ \ast  & x^t\hat D^2x+(x^tx)^2\end{array}\right),\\
		\widehat{X}^3\widetilde{Y} &= \left(\begin{array}{cc} \hat D^3+xx^t\hat D+\hat Dxx^t & \ast \\ \ast & -x^t\hat Dx\end{array}\right),\\
		\widehat{X}\widetilde{Y}\widehat{X}\widetilde{Y} &= \left(\begin{array}{cc} \hat D^2-xx^t & -\hat Dx \\ x^t \hat D  & -x^tx\end{array}\right),\\
		\widetilde{Y}^4  &= I_n,
	\end{split}
	\end{equation*}
	where $\hat D=D_0\oplus 0.$


\section{Theorem 4.2 - Remaining cases}\label{proof-rank5-extension}

\subsection{Relations $\bX\bY+\bY\bX=\mbf{0}$ and $\bY^2=\mbf 1$} \label{subsub2}

By Lemma \ref{zero-moment-general2}, the matrix $\cM_n(\beta_1,\beta_X,X)$ must have $\beta_X=0$ and hence we will write it as
$\cM_n(\beta_1,X)$. The forms of $\cM_n(\beta_1,X)$, $\cM_n(\beta_1,Y)$ are 
%
%
%
%
\begin{equation}\label{MX-form}
\begin{blockarray}{cccccccccccc}
& \mds{1}&\bX&\bX^2&\bX^3&\cdots&\bX^{2k}&\bX^{2k+1}&\cdots&\bX^n\\
\begin{block}{c(ccccccccccc)}
\mds{1}& \beta_1 & 0 & \beta_{X^2} & 0 & \cdots & \beta_{X^{2k}} & 0 & 
\cdots & c_n \beta_{X^n}\\
\bX& 0 & \beta_{X^2} & 0 & \beta_{X^4} & \cdots & 0 & \beta_{X^{2k+2}}& \cdots & 
c_{n+1}\beta_{X^{n+1}} \\
\bX^2& \beta_{X^2} & 0 & \beta_{X^4} & 0 & \cdots  & \beta_{X^{2k+2}} & 0 & \cdots&
c_n \beta_{X^{n+2}} \\
\bX^3& 0 & \beta_{X^4} & 0 & \beta_{X^6} & \cdots  & 0 & \beta_{X^{2k+4}} & \cdots & 
c_{n+1}\beta_{X^{n+3}} \\
\vdots& \vdots &&&&&&&&\vdots \\
\bX^n& c_n\beta_{X^n} & 
c_{n+1}\beta_{X^{n+1}} &
c_{n}\beta_{X^{n+2}} &  
c_{n+1}\beta_{X^{n+3}} & 
\cdots &
c_{n}\beta_{X^{n+2k}} &
c_{n+1}\beta_{X^{n+2k+1}} & 
\cdots&
\beta_{X^{2n}}\\
\end{block}
\end{blockarray},
\end{equation}
\begin{equation*}
\begin{blockarray}{cccccccccc}
& \bY& \bX\bY&\cdots&
\bX^{2k}\bY&\bX^{2k+1}\bY&\cdots&\bX^{n-1}\bY\\
\begin{block}{c(ccccccccc)}
\bY& \beta_1 & 0 & \cdots & \beta_{X^{2k}} & 0 & 
\cdots & c_{n-1} \beta_{X^n}\\
\bX\bY& 0 & \beta_{X^2}  & \cdots & 0 & \beta_{X^{2k+2}}& \cdots & 
c_{n}\beta_{X^{n+1}} \\
\vdots& \vdots &&&&&&\vdots \\
\bX^{2k}\bY& \beta_{X^{2k}} & 0  & \cdots  & \beta_{X^{4k}} & 0 & \cdots&
c_{n-1} \beta_{X^{n+2}} \\
\bX^{2k+1}\bY & 0 & \beta_{X^{2k+2}} & \cdots  & 0 & \beta_{X^{4k+2}} & \cdots & 
c_{n}\beta_{X^{n+3}} \\
\vdots& \vdots &&&&&&\vdots \\
\bX^{n-1}\bY& c_{n-1}\beta_{X^{n-1}} & 
c_{n}\beta_{X^{n}} &
\cdots &
c_{n-1}\beta_{X^{n+2k-1}} &
c_{n}\beta_{X^{n+2k}} & 
\cdots&
\beta_{X^{2n-2}}\\
\end{block}
\end{blockarray}, 
\end{equation*}
respectively, where $c_m=\frac{(-1)^{m}+1}{2}$, and $B(\beta_Y)$ has the form \eqref{B(beta-y)}.
By Lemma \ref{zero-moment-general} the nc atoms must be of the form
\eqref{form-of-nc-atoms-general}.
Hence the only way to cancel the $\beta_Y$ moment in 
$B(\beta_Y)$ is by using atoms of size 1, which are $(0,\pm 1)$.
Since we have that
$$|\beta_Y|\widetilde{\cM}_{n}^{(0,\sign(\beta_Y)1)} \preceq \gamma\widetilde{\cM}_{n}^{(0,1)}+\delta\widetilde{\cM}_{n}^{(0,-1)}$$
for every $\gamma,\delta\geq 0$ such that $\gamma-\delta=\beta_Y$ (Indeed, 
$\beta_Y\geq 0$ implies that $\sign(\beta_Y)1=1$, $\gamma\geq \beta_Y$ and hence 
$|\beta_Y|\widetilde{\cM}_{n}^{(0,\sign(\beta_Y)1)} \preceq \gamma\widetilde{\cM}_{n}^{(0,1)}$,
while 
$\beta_Y< 0$ implies that $\sign(\beta_Y)1=-1$, $\delta\geq |\beta_Y|$ and hence 
$|\beta_Y|\widetilde{\cM}_{n}^{(0,\sign(\beta_Y)1)} \preceq \delta\widetilde{\cM}_{n}^{(0,-1)}$.),
it follows that 
$\widetilde{\cM_n}(\beta_1,\beta_Y)$ admits a measure if and only if 
$$\widetilde{\cM_n}(\beta_1-\beta_Y,0)=\widetilde{\cM_n}(\beta_1,\beta_Y)-|\beta_Y|\widetilde{\cM}_{n}^{(0,\sign(\beta_Y)1)}$$
admits a measure.
Note that the existence of a measure $\cM_n(\beta_1-\beta_Y,X)$ is the truncated Hamburger moment problem.
By \cite[Theorem 3.9]{CF91}, the matrix $\cM_n(\beta_1-\beta_Y,X)$ admits a measure with size 1 atoms from $\RR$ if and only if 
$\cM_n(\beta_1-\beta_Y,X)$ is psd and recursively generated.
Now note that if $x_i$, $i=1,\ldots, k$, $k\in \NN$, are atoms in the measure for $\cM_n(\beta_1-\beta_Y,X)$
with the corresponding densities $\mu_i$,  $i=1,\ldots, k$, then
$$
\Big(\left(\begin{matrix}
0&x_i\\
x_i&0
\end{matrix}\right),
\left(\begin{matrix}
1&0\\
0&-1
\end{matrix}\right)
\Big),     \quad  i=1,\ldots,k,
$$
with densities $\mu_i$, $i=1,\ldots,k$, are atoms which represent $\widetilde{\cM_n}(\beta_1-\beta_Y,0,0)$.

\subsection{Relations $\bX\bY+\bY\bX=\mbf{0}$ and $\bY^2=\mds 1+ \bX^2$}

By Lemma \ref{zero-moment-general2}, the matrix
$\cM_n(\beta_1,\beta_X,X)$ has the form \eqref{MX-form}, $\cM_n(\beta_1,Y)$ is equal to
%
%
\begin{equation*}
\begin{blockarray}{cccccccccc}
& \bY& \bX\bY&\cdots&
\bX^{2k}\bY&\bX^{2k+1}\bY&\cdots&\bX^{n-1}\bY\\
\begin{block}{c(ccccccccc)}
\bY& \beta_1+\beta_{X^2} & 0  & \cdots & \beta_{X^{2k}}+\beta_{X^{2k+2}} & 0 & 
\cdots & \cdots\\
\bX\bY& 0 & \beta_{X^2}+\beta_{X^4}  & \cdots & 0 & \beta_{X^{2k+2}}+\beta_{X^{2k+4}}& \cdots & 
\vdots\\
\vdots& \vdots &&&&&&\vdots \\
\bX^{2k}\bY& \beta_{X^{2k}} + \beta_{X^{2k+2}}&0 &\cdots  & \beta_{X^{4k}}+\beta_{X^{4k+2}} & 0 & \cdots&
\vdots\\
\bX^{2k+1}\bY & 0 & \beta_{X^{2k+2}}+\beta_{X^{2k+4}} &\cdots & 0 & \beta_{X^{4k+2}}+\beta_{X^{4k+4}} & \cdots & 
\vdots\\
\vdots& \vdots &&&&&&\vdots \\
\bX^{n-1}\bY& \cdots&\cdots&\cdots&\cdots&\cdots&\cdots&\cdots \\
\end{block}
\end{blockarray}, 
\end{equation*}
and $B(\beta_Y)$ has the form \eqref{B(beta-y)}.
By Lemma \ref{zero-moment-general} the nc atoms must be of the form
\eqref{form-of-nc-atoms-general}.
Hence the only way to cancel the $\beta_Y$ moment in 
$B(\beta_Y)$ is by using atoms of size 1, which are $(0,\pm 1)$.
As in \S \ref{subsub2}
we argue that 
$\widetilde{\cM_n}(\beta_1,\beta_Y)$ admits a measure if and only if 
$\cM_n(\beta_1-\beta_Y,X)$ admits a measure 
with atoms from $\RR$ of size 1 
if and only if it
is psd and recursively generated.
Now note that if $x_i$, $i=1,\ldots, k$, $k\in \NN$, are atoms in the measure for 
$\cM_n(\beta_1-\beta_Y,X)$ with the corresponding densities $\mu_i$,  $i=1,\ldots, k$, then
$$
\Big(\left(\begin{matrix}
0&x_i\\
x_i&0
\end{matrix}\right),
\left(\begin{matrix}
\sqrt{1+x_i^2}&0\\
0&-\sqrt{1+x_i^2}
\end{matrix}\right)
\Big),     \quad  i=1,\ldots,k,
$$
with densities $\mu_i$,  $i=1,\ldots, k$, are atoms which represent $\widetilde{\cM_n}(\beta_1-\beta_Y,0,0)$.

\subsection{Relations $\bX\bY+\bY\bX=\mbf{0}$ and $\bY^2= \bX^2$}

By Lemma \ref{zero-moment-general2}, the matrix
$\cM_n(\beta_1,\beta_X,X)$ has the form \eqref{MX-form}, $\cM_n(\beta_1,Y)$ is equal to
%
\begin{equation*}
\begin{blockarray}{cccccccccccc}
& \bY& \bX\bY&\bX^2\bY&\bX^3\bY&\cdots&
\bX^{2k}\bY&\bX^{2k+1}\bY&\cdots&\bX^{n-1}\bY\\
\begin{block}{c(ccccccccccc)}
\bY& \beta_{X^2} & 0 & \beta_{X^4} & 0 & \cdots & \beta_{X^{2k+2}} & 0 & 
\cdots & \cdots\\
\bX\bY& 0 & \beta_{X^4} & 0 & \beta_{X^6}  & \cdots & 0 &  \beta_{X^{2k+4}}& \cdots & 
\vdots\\
\bX^2\bY& \beta_{X^4}& 0 & \beta_{X^6} & 0 & \cdots  &  \beta_{X^{2k+4}} & 0 & \cdots&
\vdots\\
\bX^3\bY & 0 & \beta_{X^6} & 0 &  \beta_{X^8} & \cdots  & 0 &  \beta_{X^{2k+6}} & \cdots & 
\vdots\\
\vdots& \vdots &&&&&&&&\vdots \\
\bX^{n-1}\bY& \cdots&\cdots&\cdots&\cdots&\cdots&\cdots&\cdots&\cdots&\cdots \\
\end{block}
\end{blockarray}, 
\end{equation*}
and $B(\beta_Y)=\bf 0$.
Note that the existence of a measure $\cM_n(\beta_1,0,X)$ is the truncated Hamburger moment problem.
By \cite[Theorem 3.9]{CF91}, the matrix $\cM_n(\beta_1,0,X)$ admits a measure with atoms from $\RR$ of size 1 if and only if it
is psd and recursively generated.
Now note that if $x_i$, $i=1,\ldots, k$, $k\in \NN$, are atoms in the measure for 
$\cM_n(\beta_1,X)$ with the corresponding densities $\mu_i$,  $i=1,\ldots, k$, then
$$
\Big(\left(\begin{matrix}
0&x_i\\
x_i&0
\end{matrix}\right),
\left(\begin{matrix}
x_i&0\\
0&-x_i
\end{matrix}\right)
\Big),     \quad  i=1,\ldots,k,
$$
with densities $\mu_i$,  $i=1,\ldots, k$, are atoms which represent $\widetilde{\cM_n}(\beta_1,0,0)$.

\end{document}